\documentclass[11pt,a4paper]{amsart}
\usepackage[a4paper,textwidth=15.2cm,textheight=25.0cm,centering]{geometry}
\IfFileExists{newtxtext.sty}{\usepackage{newtxtext}}{\usepackage[T1]{fontenc}\usepackage{lmodern}}
\usepackage{mathtools,amssymb}
\IfFileExists{newtxmath.sty}{\usepackage{newtxmath}}{}
\usepackage{microtype,xcolor,aliascnt,titlesec}
\usepackage[hidelinks]{hyperref}
\hypersetup{
  colorlinks=true,
  linkcolor=cyan,
  anchorcolor=cyan,
  citecolor=red,
  urlcolor=black,
  pdfencoding=auto,
}
\usepackage[nameinlink,capitalise,noabbrev]{cleveref}

\hypersetup{
 pdftitle={Monotone Hadwiger Theorems on Spherical and Hyperbolic Convex Sets},
 pdfauthor={Houshan Fu},
 pdfsubject={Monotone invariant valuations in spherical, elliptic, conic, and hyperbolic convexity},
 pdfkeywords={monotone valuations, Hadwiger theorem, spherical convexity,
 hyperbolic convexity, automatic continuity}
}
\title{Monotone Hadwiger Theorems on Spherical and Hyperbolic Convex Sets}
\renewenvironment{abstract}
  {\par\small\begin{center}\bfseries\abstractname\end{center}\begin{quotation}}
  {\end{quotation}}

\titleformat{\section}
  {\normalfont\large\bfseries}{\thesection}{0.8em}{}
\titleformat{\subsection}
  {\normalfont\normalsize\bfseries}{\thesubsection}{0.7em}{}
\titlespacing*{\section}{0pt}{3.0ex plus 1ex minus .2ex}{1.4ex plus .2ex}
\titlespacing*{\subsection}{0pt}{2.4ex plus .8ex minus .2ex}{0.9ex plus .2ex}

\newtheorem{theorem}{Theorem}[section]
\newaliascnt{proposition}{theorem}
\newtheorem{proposition}[proposition]{Proposition}
\aliascntresetthe{proposition}
\newaliascnt{lemma}{theorem}
\newtheorem{lemma}[lemma]{Lemma}
\aliascntresetthe{lemma}
\newaliascnt{corollary}{theorem}
\newtheorem{corollary}[corollary]{Corollary}
\aliascntresetthe{corollary}
\newcommand{\Sph}{\mathbb S}
\newcommand{\R}{\mathbb R}
\newcommand{\K}{\mathcal K}
\newcommand{\Orth}{\mathrm{O}}
\newcommand{\SO}{\mathrm{SO}}
\newcommand{\sgn}{\operatorname{sgn}}
\newcommand{\cV}{\mathcal V}
\newcommand{\pos}{\operatorname{pos}}
\newcommand{\Gr}{\operatorname{Gr}}
\newcommand{\one}{\mathbf 1}
\newcommand{\dd}{\,\mathrm d}
\newcommand{\Ks}{\mathcal K_{\mathrm s}}
\newcommand{\Ksp}{\mathcal K_{\mathrm s}^{p}}
\newcommand{\Ps}{\mathcal P_{\mathrm s}}
\newcommand{\Psp}{\mathcal P_{\mathrm s}^{p}}

\begin{document}
\begin{center}
\makeatletter
{\Large\bfseries \@title\par}
\makeatother
\medskip
Houshan Fu

\medskip
School of Mathematics and Information Science, Guangzhou University\\
Guangzhou 510006, Guangdong, P. R. China

\medskip
Email address: fuhoushan@gzhu.edu.cn
\end{center}
\medskip
\begin{abstract}
For every $n\geq1$, we classify monotone rotation-invariant real-valued
valuations on closed spherical convex sets, without assuming continuity
or measurability. On proper sets, namely those contained in an open
hemisphere, these are precisely the nonnegative linear combinations of
the normalized spherical quermassintegrals. On all closed spherical
convex sets, they are precisely the linear combinations of the spherical
intrinsic volumes with nonnegative, nondecreasing coefficients. The
representations are unique, and all such valuations are continuous and
invariant under the full orthogonal group. In hyperbolic space, an
isometry-invariant real-valued valuation on compact convex sets is
continuous if and only if it is a linear combination of the Euler
characteristic and the hyperbolic quermassintegrals. This representation
is unique. Monotonicity is equivalent to nonnegative coefficients and
implies continuity. If monotonicity is required only between nonempty
sets, the Euler coefficient is unrestricted in the proper spherical and
hyperbolic cases, whereas the classification on all closed spherical
convex sets is unchanged. We also obtain corresponding classifications
for valuations on closed convex cones that vanish at the zero cone and
monotone classifications on compact projectively convex sets contained
in an affine chart of real elliptic space.
\end{abstract}

\begin{quotation}

\end{quotation}

\section{Introduction}

Hadwiger characterized continuous rigid-motion-invariant valuations on
Euclidean convex bodies as linear combinations of intrinsic volumes
\cite{Hadwiger,KlainRota}.  McMullen established automatic continuity for
monotone translation-invariant real valuations on Euclidean convex bodies
\cite{McMullenValuations}.  Bernig and Fu proved that a continuous
translation-invariant valuation is monotone if and only if each of its
homogeneous components is monotone \cite[Theorem~2.12]{BernigFu}.
We study spherical and hyperbolic analogues, deriving the
regularity needed for classification from monotonicity.

The conic theory developed from angle-sum relations
\cite{McMullenAngles,Sommerville} and Gr\"unbaum's Grassmann angles
\cite{GrunbaumAngles}.  We use the cone convention for spherical
convexity \cite[Section~3.1]{Schneider}: the members of $\Ks(\Sph^n)$
are precisely the sets $C\cap\Sph^n$, where $C\subseteq\R^{n+1}$ is a
closed convex cone.  The zero cone corresponds to the empty set.  For
nonempty $K\in\Ks(\Sph^n)$, its positive hull is
\[
 K^\vee\coloneqq\pos K=\{t x:t\geq0,\ x\in K\},
 \qquad \varnothing^\vee\coloneqq\{0\}.
\]
Thus, $K=C\cap\Sph^n$ implies $K^\vee=C$.  A nonempty member is
\emph{proper} if it lies in an open hemisphere, equivalently, if
$K^\vee\cap(-K^\vee)=\{0\}$.  The proper members and the empty set form
$\Ksp(\Sph^n)$.  The larger family also includes antipodal pairs:
$\{u,-u\}$ corresponds to the line $\R u$.

McMullen's continuous and monotone spherical classification problems
were recorded in \cite[Problem~49]{GruberSchneider} and
\cite[Section~3.3, p.~126]{Schneider}.  Other formulations appear in
\cite[Problem~(15.5), p.~229]{McMullenSchneider}, \cite[p.~165]{GaoHugSchneider}, and
\cite[Section~4.2, Problem~2, pp.~42--43]{Glasauer}.
The simple one-dimensional case and the two-dimensional classification
were known for continuous orthogonally invariant valuations
\cite[Proposition~11.2.2 and Theorem~11.3.1]{KlainRota}.  Schneider
characterized spherical volume among simple, nonnegative, rotation-invariant
valuations on proper spherical polytopes in every dimension
\cite[Theorem~6.2]{SchneiderCurvature}, and Hack treated smooth invariant
valuations \cite[Theorem~4.1.3]{Hack}.  Knoerr proved the continuous
classification on spherical polytopes contained in open hemispheres
\cite[Theorem~A]{Knoerr}, using his Euclidean polytopal classification
\cite[Theorem~A]{KnoerrEuclidean}.  Wang and Wu obtained the continuous
classification on all closed spherical convex sets
\cite[Theorem~1.1]{WangWu}, and on proper closed spherical convex sets in
\cite[Corollary~8.3]{WangWu}.

Throughout, a valuation on a family $\mathcal A$ containing
$\varnothing$ is a map $\mu:\mathcal A\to\R$ satisfying
$\mu(\varnothing)=0$ and
$\mu(K)+\mu(L)=\mu(K\cup L)+\mu(K\cap L)$ whenever
$K,L,K\cup L,K\cap L\in\mathcal A$.  Monotonicity on nonempty sets means that
$\mu(K)\leq\mu(L)$ for $\varnothing\neq K\subseteq L$.
We use the spherical Hausdorff topology on nonempty sets and isolate
the empty set.  Denote the conic intrinsic volumes in $\R^{n+1}$ by
$v_0,\ldots,v_{n+1}$.  Following \cite[Equation~(3.20)]{Schneider} and
\cite[Section~2.3]{WangWu}, define the spherical intrinsic volumes by
\[
 v_j^{\mathrm s}(K)\coloneqq v_{j+1}(K^\vee)\quad(K\neq\varnothing),
 \qquad v_j^{\mathrm s}(\varnothing)\coloneqq0,
 \qquad 0\leq j\leq n.
\]
For monotone valuations on $\Ksp(\Sph^n)$, the spherical
quermassintegrals lead to simpler coefficient conditions.  Write $\one_A$ for
the indicator of $A$, $\Gr(\R^{n+1},n+1-j)$ for the Grassmannian of
$(n+1-j)$-dimensional linear subspaces, and $\nu_{n+1-j}$ for its
$\Orth(n+1)$-invariant probability measure.  For $0\leq j\leq n$, the
\emph{$j$th normalized spherical quermassintegral}
$U_j:\Ksp(\Sph^n)\to\R$ is defined by $U_j(\varnothing)=0$ and, for
$K\neq\varnothing$, by
\[
 U_j(K)\coloneqq\frac12\int_{\Gr(\R^{n+1},n+1-j)}
 \one_{\{K\cap L\neq\varnothing\}}\dd\nu_{n+1-j}(L).
\]
The factor $1/2$ agrees with \cite[Equation~(2.43)]{Schneider}.

\begin{theorem}[Monotone Spherical Hadwiger Theorem]
\label{thm:main}\label{thm:monotone-all}
Let $n\geq1$.
\begin{enumerate}
\renewcommand{\labelenumi}{\textup{(\roman{enumi})}}
\item An $\SO(n+1)$-invariant valuation
$\mu:\Ksp(\Sph^n)\to\R$ is monotone on nonempty sets if and only if
it has a unique representation
\[
 \mu=\sum_{j=0}^n c_jU_j,
 \qquad c_0\in\R,\quad c_1,\ldots,c_n\geq0.
\]
A valuation with this representation is monotone on $\Ksp(\Sph^n)$
if and only if $c_0\geq0$.
\item An $\SO(n+1)$-invariant valuation
$\mu:\Ks(\Sph^n)\to\R$ is monotone if and only if it has a unique
representation
\[
 \mu=\sum_{j=0}^n a_jv_j^{\mathrm s},
 \qquad 0\leq a_0\leq a_1\leq\cdots\leq a_n.
\]
\end{enumerate}
Every such valuation is continuous and $\Orth(n+1)$-invariant.
\end{theorem}

$U_0(K)=1/2$ for every nonempty $K\in\Ksp(\Sph^n)$, which explains the
unrestricted coefficient $c_0$. For valuations on $\Ks(\Sph^n)$,
monotonicity on nonempty sets also holds for comparisons with $\varnothing$.
Indeed, for $p\in K\neq\varnothing$, monotonicity and the valuation identity imply
\[
 0\leq\mu(\{p,-p\})-\mu(\{-p\})=\mu(\{p\})\leq\mu(K).
\]
Consequently, no additional assumption is needed for comparison with the empty set.

The classifications in \autoref{thm:main} also appear in conic form in
\cite[Corollary~1.2 and Theorem~1.1]{LotzMonotone}.
Under the link correspondence, part~\textup{(i)} is the pointed-cone
classification, whereas part~\textup{(ii)} is the all-cone classification
normalized to vanish at the zero cone.
Our proof uses the oriented simplex identity and signed-coning
dimension reduction of Wang and Wu
\cite[Lemma~6.2 and Sections~7--8]{WangWu}.
We establish the required measurability and prove that signed coning
preserves a class of bounded simple valuations represented by a
continuous term and finitely many bounded monotone terms.
This permits dimension reduction without assuming continuity of the
original valuation.  Restriction to great hyperspheres and two-sided
approximation then give the classification.

Klain classified continuous invariant valuations on hyperbolic polygons,
including polygons with ideal vertices
\cite[Theorem~3.2]{KlainHyperbolic}.  In higher odd dimensions, he showed
that simple continuous invariant valuations on hyperbolic polytopes,
either compact or with ideal vertices, are determined by their values on ideal
simplices \cite[Theorem~5.5]{KlainHyperbolic}.
His continuity assumption includes limits with fixed ideal vertices.
Our hyperbolic theorem treats compact convex sets in every dimension.

Bernig, Faifman, and Solanes classified invariant generalized valuations
on isotropic space forms \cite[Theorem~C]{BernigFaifmanSolanes}.
For the smooth-valuation framework, see \cite{AleskerFu}.
We work directly with compact hyperbolic convex sets, using reflection averaging
and simplex estimates, and assume no extension to generalized valuations.

We take \(\mathbb H^n\) to have sectional curvature \(-1\), and write
\(\K(\mathbb H^n)\) for its compact geodesically convex sets together
with the empty set.  We use the hyperbolic Hausdorff metric and isolate
the empty set.  Write \(\chi\) for the Euler valuation and
\(W_0^{(n)},\ldots,W_{n-1}^{(n)}\) for the quermassintegrals in the
normalization of Andrews and Wei \cite[Equation~(1.3)]{AndrewsWei}, with ambient
dimension \(n\) in place of their \(n+1\).
Thus, \(W_0^{(n)}=\operatorname{vol}_n\) and
\(W_1^{(n)}(K)=\operatorname{area}_{n-1}(\partial K)/n\) for
full-dimensional \(K\), \(n\geq2\).  Their Crofton definition is
given in \autoref{subsec:hb-crofton}.  Invariance is under the full
isometry group \(\operatorname{Isom}(\mathbb H^n)\).

\begin{theorem}[Hyperbolic Hadwiger Theorem]
\label{thm:hyperbolic}\label{thm:hyperbolic-continuous}
Let $n\geq1$ and $\mu:\K(\mathbb H^n)\to\R$ be an
isometry-invariant valuation.
\begin{enumerate}
\renewcommand{\labelenumi}{\textup{(\roman{enumi})}}
\item $\mu$ is continuous if and only if it has a unique representation
\[
 \mu=a\chi+\sum_{k=0}^{n-1}b_kW_k^{(n)},
 \qquad a,b_0,\ldots,b_{n-1}\in\R.
\]
\item $\mu$ is monotone on nonempty sets if and only if it has the
representation in \textup{(i)} with \mbox{$b_0,\ldots,b_{n-1}\geq0$}.
\item $\mu$ is monotone on $\K(\mathbb H^n)$ if and only if it has the
representation in \textup{(i)} with nonnegative coefficients.
\end{enumerate}
Every such monotone valuation is continuous.
\end{theorem}

In dimension two, \(W_0^{(2)}=A\) and \(W_1^{(2)}=P/2\), where \(A\) is area and
\(P\) is perimeter, with \(P(I)=2\operatorname{length}_{\mathbb H}(I)\)
for a segment \(I\) \cite[Section~2]{KlainHyperbolic}.
Thus, \autoref{thm:hyperbolic}\textup{(i)} describes all linear combinations of
\(\chi,P,A\).  In the monotone case, the coefficients of \(P\) and
\(A\) are nonnegative.

We state the conic formulations in
\autoref{subsec:conic-formulation} and derive the elliptic consequences in
\autoref{subsec:elliptic-corollaries}.  Our elliptic domain consists of
compact projectively convex sets contained in an affine chart of
$\mathbb{RP}^n$; see \cite[Section~9.2]{DanzerGrunbaumKlee} for background
on projective convexity.
For invariant measures and Crofton formulas in elliptic space, see
\cite[Sections~17.3--17.4]{Santalo}.
The antipodal covering induces a bijection between
isometry-invariant elliptic valuations and orthogonally invariant valuations
on proper spherical convex sets. This bijection preserves monotonicity
and continuity with respect to the elliptic and spherical Hausdorff metrics
in \autoref{lem:elliptic-lift}. Thus the continuous classification of
Wang and Wu \cite[Corollary~8.3]{WangWu} and the monotone classification
in \autoref{thm:main} yield the corresponding elliptic classifications.

\section{Preliminaries}\label{sec:preliminaries}

For spherical and conic terminology, see
\cite[Chapters~1--4]{Schneider}; for Euclidean convexity, see
\cite[Chapter~1]{SchneiderConvexBodies}.

\subsection{Spherical convex sets, cones, and topology}
Throughout, $n$ denotes the dimension of the sphere or hyperbolic space.
In the spherical setting, the corresponding cone lies in an
$(n+1)$-dimensional Euclidean space.
For the definitions below, let $W$ be any Euclidean space of dimension
$n+1$, where $n\geq0$, with origin $0$, inner product
$\langle\cdot,\cdot\rangle$, and norm $\|\cdot\|$. Its unit sphere is
\[
 \Sph(W)\coloneqq\{x\in W:\|x\|=1\}.
\]
Accordingly, $\Sph(\R^{n+1})=\Sph^n$; we also set
$\Sph(\{0\})=\varnothing$. Every linear subspace carries the induced
Euclidean structure. Orthogonal complements are taken in the ambient
Euclidean space under consideration.
Write $\Orth(W)$ for the orthogonal group of $W$ and
$\SO(W)=\{g\in\Orth(W):\det g=1\}$ for its special orthogonal group.
For $W=\R^{n+1}$, we write $\Orth(n+1)$ for $\Orth(W)$ and
$\SO(n+1)$ for $\SO(W)$.
Vectors are written in ordinary lower-case letters, and tuples of vectors
in bold italic letters. Vectors in a determinant
are its columns, and a hat over an entry means that the entry is omitted.

For $e\in \Sph(W)$, set
\[
 H_e\coloneqq\{x\in \Sph(W):\langle e,x\rangle=0\}=\Sph(W\cap e^\perp),
 \qquad
 S_e^\pm\coloneqq\{x\in \Sph(W):\mathord\pm\langle e,x\rangle>0\}.
\]
The set $H_e$ is a \emph{great hypersphere}, and $S_e^+$ and $S_e^-$ are its
complementary \emph{open hemispheres}.  Their closures are denoted by
$\overline S_e^\pm\coloneqq\{x\in \Sph(W):\mathord\pm\langle e,x\rangle\geq0\}$.
The \emph{gnomonic projection} with pole $e$ is the
homeomorphism
\[
 g_e:S_e^+\to W\cap e^\perp,\qquad
 g_e(x)\coloneqq\frac{x}{\langle e,x\rangle}-e,\qquad
 g_e^{-1}(z)\coloneqq\frac{e+z}{\|e+z\|}.
\]
The map $g_e$ identifies closed spherical convex sets contained in $S_e^+$
with compact Euclidean convex sets in $W\cap e^\perp$, and spherical polytopes
with Euclidean polytopes \cite[Section~3.2]{Schneider}.
In particular, if $U\subseteq W$ has codimension one and $e$ is a unit normal
to $U$ in $W$, then $\Sph(U)=H_e$.

For $A\subseteq W$, $\operatorname{span}A$, $\operatorname{aff}A$, and
$\operatorname{conv}A$ denote its linear, affine, and convex hulls,
respectively.  Its \emph{positive hull} is
\[
 \pos A\coloneqq
 \{t_1x_1+\cdots+t_qx_q:q\geq1,\ t_i\geq0,\ x_i\in A
 \text{ for }1\leq i\leq q\}\cup\{0\}.
\]

A \emph{convex cone} is a nonempty set $C\subseteq W$ such that
$s x+t y\in C$ whenever $x,y\in C$ and $s,t\geq0$.
Every closed convex cone $C\subseteq W$ determines the closed spherically
convex set $C\cap\Sph(W)$.  Conversely, the cone determined by a closed
spherically convex set $K\subseteq\Sph(W)$ is its positive hull
\[
 K^\vee\coloneqq\pos K\quad(K\neq\varnothing),
 \qquad \varnothing^\vee\coloneqq\{0\}.
\]
The maps $C\mapsto C\cap\Sph(W)$ and $K\mapsto K^\vee$ are inverse to each
other, with the zero cone corresponding to the empty set
\cite[Section~3.1]{Schneider}.  We denote
the family of all closed
spherically convex sets, with the empty set adjoined, by $\Ks(\Sph(W))$.
A cone $C$ is \emph{pointed} if $C\cap(-C)=\{0\}$.  A nonempty
$K\in\Ks(\Sph(W))$ is \emph{proper} if it is contained in an open hemisphere;
equivalently, $K^\vee$ is pointed \cite[Section~3.1]{Schneider}.  The family
consisting of the proper
members and the empty set is denoted by $\Ksp(\Sph(W))$.  We set
$\dim K\coloneqq\dim\operatorname{span}K^\vee-1$; in particular,
$\dim\varnothing=-1$.
A closed convex cone is \emph{polyhedral} if it is the intersection of
finitely many closed halfspaces whose boundaries contain $0$; the empty
intersection is $W$. Equivalently, it is the positive hull of a finite set, by the
Minkowski--Weyl theorem \cite[Section~2.2]{WangWu}.
A \emph{spherical polytope} is an intersection $C\cap\Sph(W)$, where $C$ is a
polyhedral cone.  The symbols $\Ps(\Sph(W))$ and $\Psp(\Sph(W))$
denote, respectively, all spherical polytopes and all proper spherical
polytopes, in both cases including $\varnothing$.  Neither notation implies
full dimensionality.

We write $d$ for Euclidean Hausdorff distance and $d_{\mathrm s}$,
$d_{\mathrm c}$, $d_{\mathbb H}$, and $d_{\mathrm e}$ for the spherical,
conic, hyperbolic, and elliptic Hausdorff distances, respectively.
Point distances use $\operatorname{dist}$ with the corresponding subscript.
The elliptic and hyperbolic distances are defined in
\autoref{subsec:elliptic-corollaries} and \autoref{sec:hyperbolic}, respectively.
The Euclidean point-to-set distance and Hausdorff distance are
\[
 \operatorname{dist}(x,A)\coloneqq\inf_{a\in A}\|x-a\|,\qquad
 d(A,B)\coloneqq\max\Big\{
 \sup_{a\in A}\operatorname{dist}(a,B),
 \sup_{b\in B}\operatorname{dist}(b,A)\Big\},
\]
where $A,B$ are nonempty compact subsets of a Euclidean space.
We write $B^n=\{x\in\R^n:\|x\|<1\}$ and
$\overline B^n=\{x\in\R^n:\|x\|\leq1\}$.

On $\Sph(W)$, the spherical point distances are
\[
 \operatorname{dist}_{\mathrm s}(x,y)\coloneqq\arccos\langle x,y\rangle,
 \qquad
\operatorname{dist}_{\mathrm s}(x,A)
 \coloneqq\inf_{a\in A}\operatorname{dist}_{\mathrm s}(x,a).
\]
For nonempty compact $K,L\subseteq\Sph(W)$, set
\[
 d_{\mathrm s}(K,L)\coloneqq\max\Big\{
 \sup_{x\in K}\operatorname{dist}_{\mathrm s}(x,L),
 \sup_{y\in L}\operatorname{dist}_{\mathrm s}(y,K)\Big\}.
\]
Extend this metric by $d_{\mathrm s}(\varnothing,\varnothing)=0$ and
$d_{\mathrm s}(\varnothing,K)=d_{\mathrm s}(K,\varnothing)=\pi$ for
$K\neq\varnothing$. Hence, the empty set is isolated.
For the family $\mathcal C(W)$ of closed convex cones in $W$, put
\[
 \overline B_W\coloneqq\{x\in W:\|x\|\leq1\},\qquad
 d_{\mathrm c}(C,D)\coloneqq d(C\cap \overline B_W,D\cap \overline B_W).
\]
We equip $\mathcal C(W)$ with the $d_{\mathrm c}$-topology and
$\Ks(\Sph(W))$ with the $d_{\mathrm s}$-topology.
The families $\Ksp(\Sph(W))$, $\Ps(\Sph(W))$, and
$\Psp(\Sph(W))$ carry the subspace topologies.
For cone and spherical set variables, convergence and continuity refer to
these topologies. Interiors of
subsets of a sphere are relative to that sphere; relative interiors
of Euclidean convex sets are taken in their affine hulls.
For a spherical convex set, relative interior means interior in
its smallest containing great subsphere.

For a closed convex cone $C\subseteq W$, define its spherical link by
\[
 \operatorname{link}(C)\coloneqq C\cap\Sph(W)\quad(C\neq\{0\}),
 \qquad \operatorname{link}(\{0\})\coloneqq\varnothing.
\]
\begin{lemma}[{\cite[Lemma~2.1]{WangWu}}]\label{lem:link-homeomorphism}
The bijection
\[
 \operatorname{link}:\mathcal C(W)\to\Ks(\Sph(W))
\]
is a homeomorphism from the conic Hausdorff topology to the spherical
Hausdorff topology.
\end{lemma}

\subsection{Valuations and quermassintegrals}

Let $\mathcal A$ be one of $\Ks(\Sph(W))$, $\Ksp(\Sph(W))$,
$\Ps(\Sph(W))$, or $\Psp(\Sph(W))$.  A real-valued \emph{valuation} on
$\mathcal A$ is a map
$\mu:\mathcal A\to\R$ such that $\mu(\varnothing)=0$ and
\[
 \mu(K)+\mu(L)=\mu(K\cup L)+\mu(K\cap L),
\]
whenever $K,L,K\cup L\in\mathcal A$; in this case
$K\cap L\in\mathcal A$ automatically.  For a subgroup $G\leq\Orth(W)$, the
valuation is \emph{$G$-invariant} if $\mu(gK)=\mu(K)$ for every $g\in G$ and
$K\in\mathcal A$.  It is \emph{simple} on $\Sph(W)$ if it vanishes on every
nonempty $K\in\mathcal A$ with $\dim K<\dim\Sph(W)$.

We call $\mu$ \emph{monotone on nonempty sets} if
\[
 \mu(K)\leq\mu(L)\qquad
 \text{whenever }K,L\in\mathcal A\text{ and }\varnothing\neq K\subseteq L,
\]
and simply \emph{monotone} if $\mu(K)\leq\mu(L)$ is required whenever
$K\subseteq L$, including $K=\varnothing$.  In particular, monotonicity includes
$0=\mu(\varnothing)\leq\mu(L)$.

We denote spherical Lebesgue measure on $\Sph^n$ by $\sigma_n$ and set
$\omega_{n+1}\coloneqq\sigma_n(\Sph^n)$.  The pushforward of $\sigma_n$ to $\Sph(W)$ under
any linear isometry $\R^{n+1}\to W$ is independent of the isometry and is again
denoted by $\sigma_n$.  Therefore, $\omega_{n+1}^{-1}\sigma_n$ is the invariant
probability measure on $\Sph(W)$.
Throughout, a completed Borel measure is the extension to the $\sigma$-algebra
generated by the Borel sets and all subsets of Borel null sets; measurability
is understood with respect to the completed measure.  For a
bounded real-valued function $f$ on a set $X$, write
$\|f\|_\infty\coloneqq\sup_{x\in X}|f(x)|$.

We define spherical quermassintegrals from conic intrinsic volumes using
Schneider's normalization.  For a polyhedral cone $C\subseteq W$,
write $\mathcal F_k(C)$ for the set of its $k$-dimensional faces and
$\operatorname{relint}F$ for the relative interior of a face $F$.  For
$x\in W$, the \emph{Euclidean metric projection} $\pi_C(x)$ is the unique point of $C$
satisfying
\[
 \|x-\pi_C(x)\|=\min_{y\in C}\|x-y\|.
\]
When $C$ is a linear subspace $L$, $\pi_L$ is orthogonal projection onto $L$.
Existence and uniqueness follow from the metric projection theorem for closed
Euclidean convex sets \cite[Section~1.2]{SchneiderConvexBodies}.
For a centered Gaussian random vector $g$ in $W$ with identity covariance,
the $k$th \emph{conic intrinsic volume} of $C$ is
\begin{equation}\label{eq:conic-gaussian}
 v_k(C)\coloneqq\mathbb P\bigl(\pi_C(g)\in\operatorname{relint}F
 \text{ for some }F\in\mathcal F_k(C)\bigr),
 \qquad 0\leq k\leq n+1.
\end{equation}
This is the probabilistic definition of conic intrinsic volumes
\cite[Equations~(2.28)--(2.30)]{Schneider}; see also the combinatorial
account in \cite[Section~2.2]{AmelunxenLotz}.
The conic intrinsic volumes extend uniquely and continuously from polyhedral
cones to all closed convex cones \cite[Section~4.2, p.~152]{Schneider}, with respect
to the conic Hausdorff topology.  We retain the notation $v_k$ for these
extensions.  Each $v_k$ is a
nonnegative, $\Orth(W)$-invariant valuation, and
\begin{equation}\label{eq:conic-total-mass}
 \sum_{k=0}^{n+1}v_k(C)=1.
\end{equation}
The valuation property, nonnegativity, orthogonal invariance, and
\eqref{eq:conic-total-mass} are collected in
\cite[Theorem~2.3.3 and Section~4.2]{Schneider}.

For $K\in\Ks(\Sph(W))\setminus\{\varnothing\}$, define its
\emph{spherical intrinsic volumes} by
\begin{equation}\label{eq:spherical-intrinsic-def}
 v_j^{\mathrm s}(K)\coloneqq v_{j+1}(K^\vee)\qquad(0\leq j\leq n),
\end{equation}
and set $v_j^{\mathrm s}(\varnothing)\coloneqq0$.  For a closed convex cone
$C\subseteq W$ and $K\in\Ks(\Sph(W))\setminus\{\varnothing\}$, set
\begin{equation}\label{eq:Udef}
 U_j(C)\coloneqq\sum_{k\geq0}v_{j+1+2k}(C),\qquad
 U_j(K)\coloneqq U_j(K^\vee),\qquad 0\leq j\leq n,
\end{equation}
where terms with index greater than $n+1$ are omitted.  Set
$U_j(\varnothing)\coloneqq0$.  The quantities $U_j(C)$ and $U_j(K)$ are
called the $j$th \emph{Grassmann angle} and the $j$th \emph{spherical
quermassintegral}, respectively; see
\cite[Sections~2.3 and~3.3]{Schneider}.

The conic and spherical intrinsic volumes and the quantities $U_j$ are
invariant under isometric enlargement of the ambient Euclidean space.
Orthogonal coordinates identify such an embedding of $C$
with $C\times\{0\}$.  Interpret intrinsic volumes outside their
dimensional range as zero.  For polyhedral cones $C,D$, the product formula
\cite[Equation~(2.34)]{Schneider} is
\[
 v_k(C\times D)=\sum_{r+s=k}v_r(C)v_s(D).
\]
Taking $D=\{0\}$, with $v_0(\{0\})=1$ and $v_s(\{0\})=0$ for $s>0$,
we obtain
\[
 v_k(C\times\{0\})=v_k(C).
\]
The same identity holds for all closed convex cones by continuity, since
$d_{\mathrm c}(C\times\{0\},D\times\{0\})=d_{\mathrm c}(C,D)$.
Definitions~\eqref{eq:spherical-intrinsic-def} and~\eqref{eq:Udef}
imply the same invariance for $v_j^{\mathrm s}$ and $U_j$.

For $0\leq k\leq\dim W$, $\Gr(W,k)$ denotes the Grassmannian of
$k$-dimensional linear subspaces of $W$, with the topology determined by
operator-norm convergence of the orthogonal projections $\pi_L$.
The definitions of $v_j^{\mathrm s}$ and $U_j$ follow
\cite[Equations~(2.68), (2.70), and~(3.20)]{Schneider}.
The spherical intrinsic volumes $v_j^{\mathrm s}$ are denoted by $V_j$ in
\cite[Section~2.3]{WangWu}.  Write $\nu_{n+1-j}$ for the
$\Orth(W)$-invariant probability measure on $\Gr(W,n+1-j)$.  If
$K\in\Ks(\Sph(W))\setminus\{\varnothing\}$ and $K^\vee$ is not a linear
subspace, then
\begin{equation}\label{eq:Ucrofton}
 U_j(K)=\frac12\int_{\Gr(W,n+1-j)}
 \one_{\{K\cap L\neq\varnothing\}}\dd\nu_{n+1-j}(L),
 \qquad 0\leq j\leq n.
\end{equation}
Formula~\eqref{eq:Ucrofton} is the conic Crofton formula
\cite[Equations~(2.43) and~(4.75)]{Schneider}.
For $n=0$, $K$ is a singleton and both sides equal $1/2$.

In particular, \eqref{eq:Ucrofton} holds for every nonempty proper $K$.
On proper spherical convex sets, the Euler characteristic
\cite[Section~3.3, p.~123]{Schneider} is the valuation
\[
 \chi_{\mathrm s}:\Ksp(\Sph(W))\to\R,\qquad
 \chi_{\mathrm s}(K)=
 \begin{cases}0,&K=\varnothing,\\1,&K\neq\varnothing.\end{cases}
 \qquad U_0=\tfrac12\chi_{\mathrm s}.
\]
The subscript distinguishes this spherical valuation from the hyperbolic
Euler valuation $\chi$ on $\K(\mathbb H^n)$.
The pushforward of $\omega_{n+1}^{-1}\sigma_n$ under $u\mapsto\R u$
is $\nu_1$, since it is an $\Orth(W)$-invariant probability measure on
$\Gr(W,1)$.  For nonempty proper $K$, a line with unit direction $u$ meets $K$ exactly
when $u\in K\cup(-K)$.  Applying \eqref{eq:Ucrofton} with $j=n$ and using
$K\cap(-K)=\varnothing$, we obtain
\[
 U_n(K)=\frac{\sigma_n(K\cup(-K))}{2\omega_{n+1}}
       =\frac{\sigma_n(K)}{\omega_{n+1}}.
\]
For $K\subseteq L$, the hitting indicators in \eqref{eq:Ucrofton} satisfy
$\one_{\{K\cap E\neq\varnothing\}}\leq\one_{\{L\cap E\neq\varnothing\}}$
for every $E\in\Gr(W,n+1-j)$.  Integration proves monotonicity on
nonempty proper sets.  Since $U_j\geq0$ and $U_j(\varnothing)=0$,
comparisons with the empty set satisfy the same inequality.

For every closed convex cone $C$, $v_k(C)=0$ when $k>\dim C$ \cite[Section~2.3]{WangWu}.  Thus
\begin{equation}\label{eq:dimension-vanishing}
 v_j^{\mathrm s}(K)=U_j(K)=0
 \qquad(K\in\Ks(\Sph(W)),\ \dim K<j\leq n)
\end{equation}
by \eqref{eq:spherical-intrinsic-def} and~\eqref{eq:Udef}, including $K=\varnothing$. If $P$ is a $j$-dimensional proper spherically convex set, then $P^\vee$ has nonempty relative interior in its $(j+1)$-dimensional linear hull. Consequently, $v_j^{\mathrm s}(P)=v_{j+1}(P^\vee)>0$ \cite[Equation~(2.32)]{Schneider}. If $E\subseteq W$ is a $k$-dimensional linear subspace, $1\leq k\leq n+1$, then $v_{k-1}^{\mathrm s}(\Sph(E))=1$ and $v_j^{\mathrm s}(\Sph(E))=0$ for $j\neq k-1$ \cite[Equation~(2.31)]{Schneider}.  With the convention $U_{n+1}=U_{n+2}=0$, \eqref{eq:Udef} yields
\begin{equation}\label{eq:U-v}
U_j(K)=v_{j+1}(K^\vee)+v_{j+3}(K^\vee)+\cdots,\qquad v_j^{\mathrm s}(K)=U_j(K)-U_{j+2}(K), \qquad 0\leq j\leq n.
\end{equation}
\subsection{The Cauchy equation and point normalization}

The one-dimensional spherical arguments reduce to additive functions of
arc length.  The following lemma derives linearity from a one-sided bound.

\begin{lemma}\label{lem:cauchy}
Let $a\in(0,\infty]$ and $f:(0,a)\to\R$ satisfy
\begin{equation}\label{eq:local-cauchy}
 f(s+t)=f(s)+f(t)\qquad(s,t,s+t\in(0,a)).
\end{equation}
If, for some nondegenerate interval $J\subseteq(0,a)$ and $M\in\R$, either
$f(t)\leq M$ for every $t\in J$ or $f(t)\geq M$ for every $t\in J$, then
$f(t)=ct$ for a unique $c\in\R$ and every $t\in(0,a)$.
\end{lemma}

\begin{proof}
We first prove that $f(h)\to0$ as $h\to0^+$.
Replacing $(f,M)$ by $(-f,-M)$ if necessary, assume that $f\leq M$ on $J$.
Choose $x_0$ and $\delta>0$ such that $[x_0-\delta,x_0+\delta]\subseteq J$.
By \eqref{eq:local-cauchy}, for $0<h<\delta$,
\[
 f(h)=f(x_0+h)-f(x_0)\leq M-f(x_0),\qquad f(h)=f(x_0)-f(x_0-h)\geq f(x_0)-M.
\]
Put $B=M-f(x_0)\geq0$.  For a positive integer $m$ and $0<h<\delta/m$,
all partial sums $h,2h,\ldots,mh$ lie in $(0,\delta)$, so
\[
 0<h<\delta/m
 \quad\Longrightarrow\quad
 |f(h)|=m^{-1}|f(mh)|\leq B/m.
\]
Thus $f(h)\to0$ as $h\to0^+$.

Fix $t_0\in(0,a)$.  For positive integers $p,r$ with $pt_0/r<a$,
repeated additivity gives
\[
 f(pt_0/r)=p f(t_0/r)=\frac p r f(t_0).
\]
For $t\in(0,a)$, choose $q_k\in\mathbb Q\cap(0,t/t_0)$ with
$q_k\to t/t_0$, so $t-q_kt_0\to0^+$.
Since $f$ is additive and $f(h)\to0$ as $h\to0^+$,
\[
 f(t)=q_kf(t_0)+f(t-q_kt_0)\longrightarrow\frac{t}{t_0}f(t_0).
\]
Thus $f(t)=ct$ for every $t\in(0,a)$, with the unique coefficient
$c=f(t_0)/t_0$.
\end{proof}

We separate the Euler term by subtracting the common point value.
For $n\geq1$ and an $\SO(n+1)$-invariant valuation
$\mu:\Ksp(\Sph^n)\to\R$ that is monotone on nonempty sets, fix
$p_0\in\Sph^n$ and define $\mu^\circ(\varnothing)\coloneqq0$ and
$\mu^\circ(K)\coloneqq\mu(K)-\mu(\{p_0\})$ for $K\neq\varnothing$.

\begin{lemma}\label{lem:normalize}
The function $\mu^\circ$ is an $\SO(n+1)$-invariant monotone valuation
that vanishes on points and is independent of the choice of $p_0$.
\end{lemma}

\begin{proof}
Since $\SO(n+1)$ acts transitively on $\Sph^n$, the value
$\mu(\{p_0\})$ is independent of $p_0$.  The Euler valuation
$\chi_{\mathrm s}$ is invariant, so
$\mu^\circ=\mu-\mu(\{p_0\})\chi_{\mathrm s}$ is an invariant valuation
vanishing on points.  To check monotonicity, take
$\varnothing\neq K\subseteq L$ and $p\in K$.  Then
\[
 \mu^\circ(L)-\mu^\circ(K)=\mu(L)-\mu(K)\geq0,
 \qquad
 \mu^\circ(K)-\mu^\circ(\varnothing)=\mu(K)-\mu(\{p\})\geq0.\qedhere
\]
\end{proof}

\section{Measurability of apex functions}\label{sec:kernels}

We establish the measurability of the apex functions used in the
signed-coning and averaging arguments.

Let $n\geq0$ and $W$ be an $(n+1)$-dimensional Euclidean space.
Denote by $\mathcal M_n(\Sph(W))$ the set of bounded real-valued
functions on the $n$-dimensional spherical polytopes in $\Sph(W)$
that are monotone under inclusion.

For an open set $D\subseteq\R^m$ and a Borel measurable function
$w:D\to(0,\infty)$, define
$\nu(A)\coloneqq\int_Aw(z)\dd z$ for every Borel set $A\subseteq D$.
Writing $|A|$ for Lebesgue measure, we have, for every such $A$,
\cite[Proposition~2.16]{Folland}
\begin{equation}\label{eq:positive-density-nullsets}
 \nu(A)=0\quad\Longleftrightarrow\quad |A|=0.
\end{equation}
Consequently, the completions of these Borel measures have the same
measurable sets \cite[Theorem~1.9]{Folland}.  A set $E\subseteq D$ belongs
to this common completed $\sigma$-algebra precisely when
$E\triangle B\subseteq Z$ for Borel sets $B,Z\subseteq D$ with $|Z|=0$,
where $E\triangle B=(E\setminus B)\cup(B\setminus E)$.

The next lemma follows from the cone-monotone measurability theorem of
Borwein, Burke, and Lewis, in the formulation for functions on open subsets in
\cite[p.~1069]{BBL}.  Their Gaussian measurability conclusion implies
Lebesgue measurability by \eqref{eq:positive-density-nullsets}, since the
standard Gaussian measure on $\R^n$ has density
$(2\pi)^{-n/2}e^{-\|z\|^2/2}>0$.

\begin{lemma}[{\cite[Corollary~7]{BBL}}]\label{lem:bbl}
For $n\geq1$, let $C\subseteq\R^n$ be a convex cone with nonempty interior,
and $D\subseteq\R^n$ be open.  If $f:D\to\R$ satisfies
\[
 f(x)\leq f(y)\qquad(x,y\in D,\ y-x\in C),
\]
then $f$ is Lebesgue measurable.
\end{lemma}

Fix a codimension-one linear subspace $U$ of $W$.  For
$x\in\Sph(W)\setminus\Sph(U)$ and $Q\in\Ks(\Sph(U))$, define
\[
 x*Q\coloneqq(Q^\vee+\R_{\geq0}x)\cap\Sph(W).
\]
We call $x*Q$ the spherical cone with apex $x$ and base $Q$
\cite[Section~7.1]{WangWu}.
The linear isomorphism
$U\times\R\to W$, $(u,t)\mapsto u+tx$, maps
$Q^\vee\times[0,\infty)$ onto $Q^\vee+\R_{\geq0}x$.
Thus this cone is closed and convex, and is polyhedral whenever
$Q^\vee$ is polyhedral.  Its linear span has dimension
$\dim\operatorname{span}Q^\vee+1$, so
\begin{equation}\label{eq:coning-dimension}
 \dim(x*Q)=\dim Q+1,
\end{equation}
including $Q=\varnothing$.
For $Q\neq\varnothing$, an equivalent expression is
\begin{equation}\label{eq:spherical-cone-param}
 x*Q=\left\{
 \frac{(1-t)y+tx}{\|(1-t)y+tx\|}:y\in Q,\ 0\leq t\leq1
 \right\}.
\end{equation}
Indeed, $Q^\vee=\{ay:a\geq0,\ y\in Q\}$, and the substitution
$t=b/(a+b)$ for $a,b\geq0$ with $a+b>0$ identifies the directions of
$ay+bx$ with those in \eqref{eq:spherical-cone-param}.
The denominator is nonzero because $x\notin U$ and $y\in\Sph(U)$.
For fixed $y\in Q$, the normalized vector in \eqref{eq:spherical-cone-param}
traces the shorter great-circle arc from $y$ to $x$ as $t$ ranges over $[0,1]$.
Thus, $x*Q$ is the union of these arcs.
For $Q=\varnothing$, the convention $\varnothing^\vee=\{0\}$ instead
implies $x*\varnothing=\{x\}$.

The argument of Wang and Wu \cite[Lemma~7.1]{WangWu} proves
the following lemma for $\dim W\geq2$.
If $\dim W=1$, then $U=\{0\}$, $Q=\varnothing$, and $x*Q=\{x\}$.
The map $(x,Q)\mapsto x*Q$ is continuous because
$\Sph(W)\times\{\varnothing\}$ is discrete.

\begin{lemma}[{\cite[Lemma~7.1]{WangWu}}]\label{lem:coning-continuity}
The map
\[
 (\Sph(W)\setminus\Sph(U))\times\Ks(\Sph(U))\to\Ks(\Sph(W)),
\qquad (x,Q)\mapsto x*Q,
\]
is continuous.
\end{lemma}

We next establish apex measurability for monotone functions on spherical polytopes,
without assuming continuity of these functions.

\begin{lemma}\label{lem:apex}
For $n\geq1$, let $W$ be an $(n+1)$-dimensional Euclidean space and $U\subseteq W$ a linear hyperplane. If $Q\in\Ps(\Sph(U))$ is full-dimensional in $\Sph(U)$ and
$\Phi\in\mathcal M_n(\Sph(W))$, then
\[
 \Sph(W)\setminus\Sph(U)\to\R,\qquad x\mapsto\Phi(x*Q),
\]
is measurable with respect to the completed measure $\sigma_n$ on each
component of $\Sph(W)\setminus\Sph(U)$.
\end{lemma}

\begin{proof}
We first establish Lebesgue measurability in gnomonic coordinates.
Choose either unit normal $e$ to $U$ in $W$ and set
$\Omega\coloneqq\{p\in\Sph(W):\langle e,p\rangle>0\}$.
Use the gnomonic coordinates $x=g_e^{-1}:U\to\Omega$,
$x(z)=(e+z)/\|e+z\|$, and put $C=Q^\vee$.
Since $Q$ spans $U$, it contains a basis $q_1,\ldots,q_n$ of $U$.
The set $\{\sum_{i=1}^n t_iq_i:t_i>0\}\subseteq C$ is nonempty and
open in $U$, so $C$ has nonempty interior.
By \eqref{eq:coning-dimension}, $\dim(x(z)*Q)=n$, so
$f(z)\coloneqq\Phi(x(z)*Q)$ is defined for every $z\in U$.
For $z\in U$ and $h\in C$,
\[
 C+\R_{\geq0}(e+z+h)
 \subseteq C+\R_{\geq0}h+\R_{\geq0}(e+z)
 =C+\R_{\geq0}(e+z).
\]
Taking spherical links and using monotonicity of $\Phi$, we obtain
\[
 x(z+h)*Q\subseteq x(z)*Q,\qquad f(z+h)\leq f(z).
\]
By \autoref{lem:bbl} applied to $-f$, $f$ is Lebesgue measurable.

To transfer this measurability to $\Omega$, write
$(x^*\sigma_n)(A)\coloneqq\sigma_n(x(A))$
for Borel sets $A\subseteq U$.
The change-of-variables formula in \cite[Lemma~3.2.3]{Schneider} states
\[
 |g_e(B)|=\int_B\langle e,p\rangle^{-(n+1)}\dd\sigma_n(p)
 \qquad(B\subseteq\Omega\text{ Borel}).
\]
Taking $B=x(A)$, with $g_e\circ x=\operatorname{id}_U$ and
$\langle e,x(z)\rangle=(1+\|z\|^2)^{-1/2}$, gives
\[
 |A|=\int_{x(A)}\langle e,p\rangle^{-(n+1)}\dd\sigma_n(p)
 =\int_A(1+\|z\|^2)^{(n+1)/2}\dd(x^*\sigma_n)(z).
\]
Equivalently, since the density is strictly positive,
\begin{equation}\label{eq:spherical-pullback-measure}
 (x^*\sigma_n)(A)
 =\int_A(1+\|z\|^2)^{-(n+1)/2}\dd z.
\end{equation}
For $c\in\R$, choose Borel sets $B_c,Z_c\subseteq U$ with
$f^{-1}((c,\infty))\triangle B_c\subseteq Z_c$ and $|Z_c|=0$.
The homeomorphism $x$ maps both sets to Borel sets, and
\eqref{eq:spherical-pullback-measure} gives
\[
 \{p\in\Omega:\Phi(p*Q)>c\}\triangle x(B_c)
 \subseteq x(Z_c),\qquad \sigma_n(x(Z_c))=0.
\]
Thus $p\mapsto\Phi(p*Q)$ is measurable with respect to the completed
measure $\sigma_n$ on $\Omega$.
\end{proof}

\section{Signed coning and finite cutting}\label{sec:coning}

To handle signed averages, we use a regularity class that permits both
positive and negative coefficients.  Its decomposition makes the apex
integrals measurable without requiring the valuation itself to be continuous.

For $n\geq0$, $\cV_n$ consists of the bounded simple valuations
$\lambda:\Ps(\Sph^n)\to\R$ admitting a representation
\begin{equation}\label{eq:regularity-class}
 \lambda(P)=\nu(P)+\sum_{\alpha=1}^q a_\alpha\Phi_\alpha(P)
 \qquad(\dim P=n),
\end{equation}
where $q\geq0$, $a_\alpha\in\R$, $\nu$ is a bounded continuous valuation
on $\Ps(\Sph^n)$, and
$\Phi_\alpha\in\mathcal M_n(\Sph^n)$.
The functions $\Phi_\alpha$ need not be defined on lower-dimensional
spherical polytopes.  For an $(n+1)$-dimensional Euclidean space $W$,
the analogous class on $\Sph(W)$ is defined by orthogonally identifying
$W$ with $\R^{n+1}$.

For $n\geq1$, fix an $(n+1)$-dimensional Euclidean space $W$,
a linear hyperplane $U\subseteq W$, and a component $\Omega$ of
$\Sph(W)\setminus\Sph(U)$.
For $\Phi\in\mathcal M_n(\Sph(W))$, define
\begin{equation}\label{eq:J}
 (J_\Omega\Phi)(Q)\coloneqq
 \frac1{\omega_{n+1}}\int_\Omega\Phi(x*Q)\dd\sigma_n(x)
 \qquad(\dim Q=n-1).
\end{equation}

\begin{lemma}\label{lem:kernelrec}
For $n\geq1$, $J_\Omega\Phi$ is well defined and satisfies
\[
 \|J_\Omega\Phi\|_\infty\leq\tfrac12\|\Phi\|_\infty,\qquad   J_\Omega\Phi\in\mathcal M_{n-1}(\Sph(U)).
\]
\end{lemma}

\begin{proof}
For well-definedness, take $Q\in\Ps(\Sph(U))$ with $\dim Q=n-1$.
Then
$\dim(x*Q)=n$ for every $x\in\Omega$ by \eqref{eq:coning-dimension}.
The function $x\mapsto\Phi(x*Q)$ is measurable by \autoref{lem:apex}
and satisfies $|\Phi(x*Q)|\leq\|\Phi\|_\infty$.
Since $\sigma_n(\Omega)<\infty$, the integral in \eqref{eq:J}
converges absolutely. Thus $J_\Omega\Phi$ is well defined, and
\[
 \begin{aligned}
 \|J_\Omega\Phi\|_\infty
 &=\sup_{\substack{Q\in\Ps(\Sph(U))\\ \dim Q=n-1}}
   \left|\frac1{\omega_{n+1}}
   \int_\Omega\Phi(x*Q)\dd\sigma_n(x)\right|\\
 &\leq\frac1{\omega_{n+1}}
   \int_\Omega\|\Phi\|_\infty\dd\sigma_n(x)\\
 &=\frac{\sigma_n(\Omega)}{\omega_{n+1}}\|\Phi\|_\infty
   =\tfrac12\|\Phi\|_\infty.
 \end{aligned}
\]
To verify monotonicity, take full-dimensional $Q_1\subseteq Q_2$ in $\Sph(U)$.
Then
$x*Q_1\subseteq x*Q_2$ for every $x\in\Omega$.
By monotonicity of $\Phi$,
\[
 \begin{aligned}
 (J_\Omega\Phi)(Q_1)
 &=\frac1{\omega_{n+1}}\int_\Omega\Phi(x*Q_1)\dd\sigma_n(x)\\
 &\leq\frac1{\omega_{n+1}}\int_\Omega\Phi(x*Q_2)\dd\sigma_n(x)
 =(J_\Omega\Phi)(Q_2).
 \end{aligned}
\]
Hence $J_\Omega\Phi\in\mathcal M_{n-1}(\Sph(U))$.
\end{proof}

For dimension reduction, we combine the two hemispherical averages with
opposite signs.
Fix an $n$-dimensional subspace $W\subseteq\R^{n+1}$ and a unit normal
$e$ to $W$.  For an $\SO(n+1)$-invariant $\lambda\in\cV_n$ and
$Q\in\Ps(\Sph(W))$, put
\begin{equation}\label{eq:T}
 (T_{W,e}\lambda)(Q)\coloneqq
 \frac1{\omega_{n+1}}\int_{\Sph^n\setminus\Sph(W)}
 \sgn\langle x,e\rangle\,\lambda(x*Q)\dd\sigma_n(x).
\end{equation}
Here $\sgn$ denotes the sign function, with $\sgn(0)=0$.
For continuous valuations on $\Ks(\Sph^n)$, Wang and Wu construct the corresponding transform in \cite[Proposition~7.2]{WangWu}. The next proposition establishes the required properties for
$\lambda\in\cV_n$ on $\Ps(\Sph^n)$.

\begin{proposition}\label{prop:T}
For $n\geq1$, let $\lambda\in\cV_n$ be $\SO(n+1)$-invariant.
The transform $T_{W,e}\lambda$ is a bounded simple valuation in $\cV_{n-1}$
with $\|T_{W,e}\lambda\|_\infty\leq\|\lambda\|_\infty$.
For every $h\in\Orth(W)$,
\begin{equation}\label{eq:signed-coning-equivariance}
 (T_{W,e}\lambda)(hQ)=(\det h)(T_{W,e}\lambda)(Q).
\end{equation}
In particular, it is $\SO(W)$-invariant.
\end{proposition}

\begin{proof}
We first show that the transform is well defined, bounded and simple.
Fix a representation \eqref{eq:regularity-class} of $\lambda$.
For $\dim Q=n-1$, each $x\mapsto\Phi_\alpha(x*Q)$ is integrable
on both hemispheres by \autoref{lem:kernelrec}, while
$x\mapsto\nu(x*Q)$ is bounded and continuous by
\autoref{lem:coning-continuity}.
For $\dim Q<n-1$, including $Q=\varnothing$, \eqref{eq:coning-dimension}
and simplicity imply
\[
 \dim(x*Q)=\dim Q+1<n,\qquad
 \lambda(x*Q)=0,\qquad (T_{W,e}\lambda)(Q)=0.
\]
Thus \eqref{eq:T} is well defined and satisfies
\[
 |(T_{W,e}\lambda)(Q)|
 \leq\frac1{\omega_{n+1}}
       \int_{\Sph^n\setminus\Sph(W)}|\lambda(x*Q)|\dd\sigma_n(x)
 \leq\|\lambda\|_\infty.
\]

To represent $T_{W,e}\lambda$ as in \eqref{eq:regularity-class}, define
\[
 \widetilde\nu(Q)\coloneqq
 \frac1{\omega_{n+1}}\int_{\Sph^n\setminus\Sph(W)}
       \sgn\langle x,e\rangle\,[\nu(x*Q)-\nu(\{x\})]\dd\sigma_n(x).
\]
Subtracting $\nu(\{x\})$ ensures $\widetilde\nu(\varnothing)=0$,
since $x*\varnothing=\{x\}$.  Also,
$\|\widetilde\nu\|_\infty\leq2\|\nu\|_\infty$.
For $Q_1,Q_2,Q_1\cup Q_2\in\Ps(\Sph(W))$ and $x\notin\Sph(W)$,
uniqueness of the decomposition in $W\oplus\R x$ gives
\begin{equation}\label{eq:coning-set-identities}
 x*(Q_1\cup Q_2)=(x*Q_1)\cup(x*Q_2)\quad\text{and}\quad x*(Q_1\cap Q_2)=(x*Q_1)\cap(x*Q_2).
\end{equation}
If $Q_1\cap Q_2=\varnothing$, both sides of the second identity equal
$\{x\}$. Applying the valuation identities for $\lambda$ and $\nu$ to
$x*Q_1,x*Q_2$ and integrating yields
\[
 \eta(Q_1)+\eta(Q_2)=\eta(Q_1\cup Q_2)+\eta(Q_1\cap Q_2),
 \qquad \eta\in\{T_{W,e}\lambda,\widetilde\nu\}.
\]
For continuity of $\widetilde\nu$, suppose that $Q_i\to Q$ in $d_{\mathrm s}$.
Then $\nu(x*Q_i)\to\nu(x*Q)$ for every
$x\notin\Sph(W)$ by \autoref{lem:coning-continuity}. Since the differences are bounded by
$2\|\nu\|_\infty$, dominated convergence implies
\[
 |\widetilde\nu(Q_i)-\widetilde\nu(Q)|
 \leq\frac1{\omega_{n+1}}\int_{\Sph^n\setminus\Sph(W)}
       |\nu(x*Q_i)-\nu(x*Q)|\dd\sigma_n(x)\longrightarrow0.
\]
With
\[
 c\coloneqq\frac1{\omega_{n+1}}\int_{\Sph^n\setminus\Sph(W)}
       \sgn\langle x,e\rangle\,\nu(\{x\})\dd\sigma_n(x),
\]
splitting \eqref{eq:T} over $S_e^+$ and $S_e^-$ gives, for $\dim Q=n-1$,
\begin{equation}\label{eq:T-regularity}
 T_{W,e}\lambda
 =\widetilde\nu+c\one+
   \sum_{\alpha=1}^q a_\alpha
       (J_{S_e^+}\Phi_\alpha-J_{S_e^-}\Phi_\alpha).
\end{equation}
By \autoref{lem:kernelrec}, each $J_{S_e^\pm}\Phi_\alpha$ belongs to
$\mathcal M_{n-1}(\Sph(W))$, as does the constant function $\one$.
Since $\widetilde\nu$ is a bounded continuous valuation,
\eqref{eq:T-regularity} proves $T_{W,e}\lambda\in\cV_{n-1}$.

For equivariance \eqref{eq:signed-coning-equivariance}, take $h\in\Orth(W)$
and define $g|_W=h$ and $ge=(\det h)e$.
Then $\det g=(\det h)^2=1$, so $g\in\SO(n+1)$.
Substitution of $x=gy$, using invariance of $\sigma_n$ and $\lambda$,
yields
\begin{align*}
 (T_{W,e}\lambda)(hQ)
 &=\frac1{\omega_{n+1}}\int_{\Sph^n\setminus\Sph(W)}
   \sgn\langle gy,e\rangle\,\lambda((gy)*(hQ))\dd\sigma_n(y)\\
 &=\frac{\det h}{\omega_{n+1}}
   \int_{\Sph^n\setminus\Sph(W)}
       \sgn\langle y,e\rangle\,\lambda(g(y*Q))\dd\sigma_n(y)\\
 &=(\det h)(T_{W,e}\lambda)(Q).\qedhere
\end{align*}
\end{proof}

Finite cutting transfers relations between indicator functions to simple
valuations.
Both $\Ps(\Sph^n)$ and $\Psp(\Sph^n)$ are closed under intersection
with closed hemispheres.  For a simple valuation $\rho$ on either class,
a spherical polytope $P$ in its domain, and a great hypersphere $H$ with closed
hemispheres $H^\pm$, the valuation identity and simplicity give
\begin{equation}\label{eq:spherical-single-cut}
 \begin{aligned}
 \rho(P)&=\rho(P\cap H^+)+\rho(P\cap H^-)-\rho(P\cap H)\\
 &=\rho(P\cap H^+)+\rho(P\cap H^-).
 \end{aligned}
\end{equation}
If $\mathcal C(P)$ denotes the full-dimensional cells obtained from $P$
by finitely many such cuts, iterating~\eqref{eq:spherical-single-cut} yields
\cite[Lemma~2.5 and Equation~(2.1)]{WangWu}
\begin{equation}\label{eq:finite-cut-sum}
 \rho(P)=\sum_{C\in\mathcal C(P)}\rho(C).
\end{equation}
\begin{lemma}\label{lem:cutting}
For $n\geq1$, let $\rho:\mathcal A\to\R$ be a simple valuation, where
$\mathcal A$ is either $\Ps(\Sph^n)$ or $\Psp(\Sph^n)$.
If $P_1,\ldots,P_q\in\mathcal A$, $\alpha_1,\ldots,\alpha_q\in\R$,
and
\[
 \sum_{i=1}^q\alpha_i\one_{P_i}(x)=0
 \qquad(x\in\Sph^n\setminus\Sigma)
\]
for a finite union $\Sigma$ of great hyperspheres, then
$\sum_{i=1}^q\alpha_i\rho(P_i)=0$.
\end{lemma}

\begin{proof}
Choose finitely many great hyperspheres whose union contains $\Sigma$
and all $\partial P_i$, including the coordinate great hyperspheres.
The closures of the components of their complement form a family
$\mathcal C$ of full-dimensional proper cells.
For $x_C\in\operatorname{int}C$, constancy of $\one_{P_i}$ on
$\operatorname{int}C$ and closedness of $P_i$ imply
\[
 x_C\in P_i\ \Longrightarrow\ C\subseteq P_i,\qquad
 x_C\notin P_i\ \Longrightarrow\ \dim(C\cap P_i)<n.
\]
Thus simplicity and \eqref{eq:finite-cut-sum} give
\[
 \rho(P_i)=\sum_{C\in\mathcal C}\one_{P_i}(x_C)\rho(C),
 \qquad 1\leq i\leq q.
\]
Hence
\[
 \sum_i\alpha_i\rho(P_i)
 =\sum_{C\in\mathcal C}
   \left(\sum_i\alpha_i\one_{P_i}(x_C)\right)\rho(C)=0.
\]
The last equality holds because $x_C\notin\Sigma$.
\end{proof}

To pass from spherical simplices to proper spherical polytopes, we use a face-to-face
triangulation.  A finite family of spherical simplices is \emph{face-to-face} if the intersection of any two spherical  simplices in the family is empty or a face of both. In dimension one, a full-dimensional proper spherical polytope is itself a nondegenerate arc.

\begin{lemma}[{\cite[Lemma~8.1]{WangWu}}]\label{lem:triangulation}
Every full-dimensional proper spherical polytope in $\Sph^n$ with $n\geq1$
admits a finite face-to-face triangulation into nondegenerate proper
spherical $n$-simplices.
\end{lemma}

\section{Vanishing of simple valuations}

The induction step recovers simplex values from apex integrals and
then relates those integrals to signed coning in one lower dimension.
We begin with the oriented simplex identity underlying this recovery.

For $p_0,\ldots,p_k\in\Sph^n$, write
\[
 \Delta(p_0,\ldots,p_k)=\pos\{p_0,\ldots,p_k\}\cap\Sph^n,
\]
also when the generators are dependent.

Call $p_0,\ldots,p_{n+1}\in\Sph^n$ \emph{in general position} if every
$n+1$ of the vectors are linearly independent.
For vectors in $\R^{n+1}$, $\det(p_0,\ldots,p_n)$ denotes the determinant
of the matrix with columns $p_0,\ldots,p_n$ in the standard orientation.
Fix $n\geq1$ and a simple valuation $\rho:\Ps(\Sph^n)\to\R$ with
$\rho(\Sph^n)=0$.  Define its \emph{oriented simplex cochain} by
\begin{equation}\label{eq:oriented-cochain}
 c_\rho(p_0,\ldots,p_n)\coloneqq
 \sgn\det(p_0,\ldots,p_n)\,
 \rho(\Delta(p_0,\ldots,p_n)).
\end{equation}

For a general-position tuple $(p_0,\ldots,p_{n+1})$, put $\Delta_i=\Delta(p_0,\ldots,\widehat p_i,\ldots,p_{n+1})$ and $d_i=(-1)^i\det(p_0,\ldots,\widehat p_i,\ldots,p_{n+1})$. Wang and Wu \cite[Lemma~6.2, Equation~(6.1)]{WangWu} proved that 
\begin{equation}\label{eq:spherical-indicator-identity} 
\sum_{i=0}^{n+1}\sgn(d_i)\one_{\Delta_i}=\gamma\one_{\Sph^n}
\end{equation}
outside a finite union of great hyperspheres, where $\gamma=1$ if all $d_i>0$, $\gamma=-1$ if all $d_i<0$, and $\gamma=0$ otherwise. Their proof also applies when $n=1$. By \autoref{lem:cutting},
\[
 \sum_{i=0}^{n+1}\sgn(d_i)\rho(\Delta_i)
 =\gamma\rho(\Sph^n)=0.
\]
Since $(-1)^ic_\rho(p_0,\ldots,\widehat p_i,\ldots,p_{n+1})=\sgn(d_i)\rho(\Delta_i)$, this is the following identity.

\begin{lemma}\label{lem:gpcocycle}
For every general-position tuple $(p_0,\ldots,p_{n+1})$,
\[
 \sum_{i=0}^{n+1}(-1)^i
 c_\rho(p_0,\ldots,\widehat p_i,\ldots,p_{n+1})=0.
\]
\end{lemma}

We use the averaging construction of Wang and Wu
\cite[Equation~(6.3) and Lemma~6.4]{WangWu}.
The next lemma uses apex measurability and boundedness in place of
their continuity assumption.

\begin{lemma}\label{lem:averaged-primitive}
For $n\geq1$, suppose that $\lambda\in\cV_n$ satisfies $\lambda(\Sph^n)=0$.
For every $p_0,\ldots,p_{n-1}\in\Sph^n$, the integral
\[
 b_\lambda(p_0,\ldots,p_{n-1})\coloneqq
 \frac1{\omega_{n+1}}\int_{\Sph^n}
 c_\lambda(x,p_0,\ldots,p_{n-1})\dd\sigma_n(x)
\]
converges absolutely.  For linearly independent $p_0,\ldots,p_n$,
\begin{equation}\label{eq:cdb}
 c_\lambda(p_0,\ldots,p_n)
 =\sum_{i=0}^n(-1)^i
    b_\lambda(p_0,\ldots,\widehat p_i,\ldots,p_n).
\end{equation}
\end{lemma}

\begin{proof}
We first check absolute convergence.
For a linearly dependent $(p_0,\ldots,p_{n-1})$, the determinant and hence
the integrand vanish.  Otherwise put
$W=\operatorname{span}\{p_0,\ldots,p_{n-1}\}$ and
$Q=\Delta(p_0,\ldots,p_{n-1})$.
For $x\notin\Sph(W)$, substituting \eqref{eq:regularity-class} into
\eqref{eq:oriented-cochain} yields
\[
 c_\lambda(x,p_0,\ldots,p_{n-1})
 =\sgn\det(x,p_0,\ldots,p_{n-1})
   \left(\nu(x*Q)+\sum_{\alpha=1}^q a_\alpha\Phi_\alpha(x*Q)\right).
\]
The sign is constant on each component, the $\nu$-term is continuous by
\autoref{lem:coning-continuity}, and the $\Phi_\alpha$-terms are measurable
by \autoref{lem:apex}.
The integrand is zero on $\Sph(W)$ and bounded in absolute value by
$\|\lambda\|_\infty$, which proves absolute integrability.

For the identity \eqref{eq:cdb}, fix linearly independent $p_0,\ldots,p_n$ and set
$W_i=\operatorname{span}\{p_0,\ldots,\widehat p_i,\ldots,p_n\}$.
The tuple $(x,p_0,\ldots,p_n)$ is in general position outside
\[
\Sigma=\bigcup_{i=0}^n\Sph(W_i).
\]
Each $\Sph(W_i)$ is a great hypersphere, so $\sigma_n(\Sigma)=0$.
By \autoref{lem:gpcocycle},
\[
 c_\lambda(p_0,\ldots,p_n)
 =\sum_{i=0}^n(-1)^i
 c_\lambda(x,p_0,\ldots,\widehat p_i,\ldots,p_n)
 \qquad(x\notin\Sigma).
\]
Integrating the absolutely integrable terms against
$\omega_{n+1}^{-1}\sigma_n$ proves \eqref{eq:cdb}.
\end{proof}

The averaged simplex identity \eqref{eq:cdb} connects the original
valuation to signed coning.  By \autoref{prop:T}, the transforms remain
in the class defined by \eqref{eq:regularity-class} in one lower dimension,
so the vanishing argument proceeds by
induction without a continuity assumption.  This adapts the argument of
\cite[Corollary~7.3, Proposition~7.4, and Proposition~8.2]{WangWu}.

\begin{proposition}\label{prop:vanishing}
For $n\geq1$, let $\lambda\in\cV_n$ be $\SO(n+1)$-invariant.
If $\lambda(\Sph^n)=0$, then $\lambda=0$.
\end{proposition}

\begin{proof}
We argue by induction on $n$.
For $n=1$, $\lambda$ vanishes on points by simplicity.
By rotation invariance, $f(t)=\lambda(I_t)$ depends only on the length
$t\in(0,\pi)$ of a proper closed arc.
The valuation identity on adjacent arcs reads
$f(s+t)=f(s)+f(t)$ whenever $s,t>0$ and $s+t<\pi$.
Boundedness and \autoref{lem:cauchy} yield $f(t)=at$.
Cutting $\Sph^1$ by the two coordinate lines and using simplicity, we obtain
\[
 0=\lambda(\Sph^1)=4f(\pi/2)=2\pi a.
\]
Hence $a=0$, and $\lambda$ vanishes on every proper spherical polytope.
Cutting an arbitrary spherical polytope into proper cells and applying
\eqref{eq:finite-cut-sum} extends this conclusion to all of $\Ps(\Sph^1)$.

Now assume $n\geq2$ and the result in dimension $n-1$.
We show that the signed-coning transforms vanish, then recover $\lambda$
from the averaged primitive.
For an $n$-dimensional subspace $W$ and a unit normal $e$,
$T_{W,e}\lambda$ is an $\SO(W)$-invariant member of $\cV_{n-1}$
by \autoref{prop:T}.
Choose $h\in\Orth(W)$ with $\det h=-1$.
Since $h\Sph(W)=\Sph(W)$, \eqref{eq:signed-coning-equivariance} implies
\[
 (T_{W,e}\lambda)(\Sph(W))
 =-(T_{W,e}\lambda)(\Sph(W))=0.
\]
Consequently, $T_{W,e}\lambda=0$ by induction.

To relate these transforms to $b_\lambda$, take independent
$p_0,\ldots,p_{n-1}$ and set
$W=\operatorname{span}\{p_0,\ldots,p_{n-1}\}$.
Choose the unit normal $e$ to $W$ with $\det(e,p_0,\ldots,p_{n-1})>0$.
The orthogonal decomposition $x=w+\langle x,e\rangle e$, $w\in W$, gives
\[
 \det(x,p_0,\ldots,p_{n-1})
 =\langle x,e\rangle\det(e,p_0,\ldots,p_{n-1}).
\]
Consequently,
\[
 \sgn\det(x,p_0,\ldots,p_{n-1})=\sgn\langle x,e\rangle\quad\text{and}\quad\Delta(x,p_0,\ldots,p_{n-1})=x*\Delta(p_0,\ldots,p_{n-1}),
\]
for $x\notin\Sph(W)$.
On $\Sph(W)$ the determinant is zero, so
\[
 b_\lambda(p_0,\ldots,p_{n-1})
 =(T_{W,e}\lambda)\bigl(\Delta(p_0,\ldots,p_{n-1})\bigr)=0.
\]
If $p_0,\ldots,p_{n-1}$ are dependent, then
$\det(x,p_0,\ldots,p_{n-1})=0$ for every $x$; hence
$c_\lambda(x,p_0,\ldots,p_{n-1})=0$ and
$b_\lambda(p_0,\ldots,p_{n-1})=0$.
Substitution in \eqref{eq:cdb} yields $c_\lambda=0$ on independent
$(n+1)$-tuples, hence $\lambda$ vanishes on nondegenerate proper
spherical $n$-simplices.  By simplicity, \autoref{lem:triangulation}, and
\autoref{lem:cutting}, it vanishes on all proper spherical polytopes.
For arbitrary $P\in\Ps(\Sph^n)$, let $\mathcal C(P)$ be the
full-dimensional proper cells contained in $P$ in the refinement
from the proof of \autoref{lem:cutting}.  Formula~\eqref{eq:finite-cut-sum} yields
\[
 \lambda(P)=\sum_{C\in\mathcal C(P)}\lambda(C)=0.
\]
\end{proof}

\section{Spherical classifications}\label{sec:spherical-classifications}
We first classify valuations on proper spherical polytopes and then pass to proper closed spherical convex sets.  Then the fixed-fan extension  determines the values on all closed spherical convex sets.
\subsection{Classification on proper spherical polytopes}
To apply \autoref{prop:vanishing}, we need an extension from proper to
arbitrary spherical polytopes.
Wang and Wu \cite[Lemmas~3.1--3.2]{WangWu} construct the unique
valuation extension from proper to all closed spherical convex sets
using a fixed finite fan.  Their proof uses only finite cuts and
orthogonal images, which preserve spherical polytopes.

Fix an orthonormal basis $e_1,\ldots,e_{n+1}$.  For
$\tau\in\{-1,0,1\}^{n+1}\setminus\{0\}$ define
\[
 R_\tau\coloneqq\{x\in\Sph^n:
  \tau_i\langle x,e_i\rangle\geq0\ \text{if }\tau_i\neq0,\quad
  \langle x,e_i\rangle=0\ \text{if }\tau_i=0\},
\]
and $z(\tau)\coloneqq\#\{i:\tau_i=0\}$.  Each $R_\tau$ is compactly
contained in an open hemisphere because
\[
 \left\langle x,\sum_{\tau_i\neq0}\tau_i e_i\right\rangle
 =\sum_{\tau_i\neq0}|\langle x,e_i\rangle|>0
 \qquad(x\in R_\tau).
\]

For a valuation $\mu$ on either $\Psp(\Sph^n)$ or $\Ksp(\Sph^n)$,
define $\overline\mu$ on $\Ps(\Sph^n)$ or $\Ks(\Sph^n)$, respectively, by
\begin{equation}\label{eq:fan}
 \overline\mu(K)\coloneqq
 \sum_{\tau\neq0}(-1)^{z(\tau)}\mu(K\cap R_\tau).
\end{equation}
\begin{lemma}[{\cite[Lemma~3.2]{WangWu}}]\label{lem:fan}
Formula~\eqref{eq:fan} defines the unique valuation extension of $\mu$.
The extension preserves $\SO(n+1)$-invariance and simplicity.
\end{lemma}

The restriction to great hyperspheres and subtraction of intrinsic-volume
terms in the next proof follow Wang and Wu
\cite[proof of Proposition~4.2]{WangWu}.

\begin{proposition}\label{prop:polyhedral-classification}
For $n\geq1$, let $\mu:\Ksp(\Sph^n)\to\R$ be an
$\SO(n+1)$-invariant valuation that is monotone on nonempty sets.
Then there exist unique $c_0,\ldots,c_n\in\R$ such that
\[
 \mu=\sum_{j=0}^n c_jU_j
 \qquad\text{on }\Psp(\Sph^n).
\]
\end{proposition}

\begin{proof}
Fix $p\in\Sph^n$ and put
$\mu^\circ=\mu-\mu(\{p\})\chi_{\mathrm s}$.
By \autoref{lem:normalize}, $\mu^\circ$ is invariant, monotone including
comparisons with the empty set, and zero on points.
Since $\chi_{\mathrm s}=2U_0$ and
$v_j^{\mathrm s}=U_j-U_{j+2}$ by \eqref{eq:U-v}, it suffices for existence to
express $\mu^\circ$ as a linear combination of the spherical intrinsic
volumes.  We proceed by induction on $n$.

For $n=1$, $f(t)\coloneqq\mu^\circ(I_t)$ depends only on the length
$t\in(0,\pi)$ of a proper closed arc $I_t$, by rotation invariance.
The valuation identity for adjacent arcs, whose common endpoint has value zero, is
$f(s+t)=f(s)+f(t)$ for $s,t>0$ with $s+t<\pi$.
Comparison of nested arcs and the empty set gives
\[
 0\leq f(t)\leq f(\pi/2)\qquad(0<t<\pi/2).
\]
Thus \autoref{lem:cauchy} implies $f(t)=bt$ for some $b\in\R$.
Since $v_1^{\mathrm s}(I_t)=t/(2\pi)$ and both valuations vanish on
points and the empty set, $\mu^\circ=2\pi b\,v_1^{\mathrm s}$ on
$\Psp(\Sph^1)$.

Assume $n\geq2$ and the result in dimension $n-1$.  Fix an
$n$-dimensional linear subspace $W\subsetneq\R^{n+1}$.
Every $h\in\SO(W)$ extends by the identity on $W^\perp$ to an element
of $\SO(n+1)$.  Hence, the restriction of $\mu^\circ$ to
$\Ksp(\Sph(W))$ satisfies the induction hypothesis.
By induction, \eqref{eq:U-v}, and invariance under enlargement of the
ambient space, there are coefficients $a_0,\ldots,a_{n-1}\in\R$ such that
\[
 \mu^\circ(Q)=\sum_{j=0}^{n-1}a_jv_j^{\mathrm s}(Q)
 \qquad(Q\in\Psp(\Sph(W))).
\]
Define $\nu=\sum_{j=0}^{n-1}a_jv_j^{\mathrm s}$ on $\Ks(\Sph^n)$ and
$\eta=\mu^\circ-\nu$ on $\Psp(\Sph^n)$.
The valuation $\nu$ is continuous and invariant.
For each nonempty $P\in\Psp(\Sph^n)$ with $\dim P<n$, choose
$g\in\SO(n+1)$ such that $gP\subseteq\Sph(W)$.  Then
\[
 \eta(P)=\mu^\circ(gP)-\nu(gP)=0.
\]
Thus, $\eta$ is simple and invariant, and its extension
$\overline\eta$ from \autoref{lem:fan} has the same properties.

To apply \autoref{prop:vanishing}, set $a_n=\overline\eta(\Sph^n)$ and
$\lambda=\overline\eta-a_n v_n^{\mathrm s}$. For $\dim P<n$, we have
$v_n^{\mathrm s}(P)=0$ by  \eqref{eq:dimension-vanishing}.
Also, $v_n^{\mathrm s}(\Sph^n)=1$.
Thus $\lambda$ is simple and invariant, and $\lambda(\Sph^n)=0$.
It remains to verify that $\lambda\in\cV_n$.
For full-dimensional $P\in\Ps(\Sph^n)$, define
$\Phi_\tau(P)\coloneqq\mu^\circ(P\cap R_\tau)$ for each $\tau\neq0$.
The spherical polytope $R_\tau$ is proper, and monotonicity of $\mu^\circ$ implies
\[
 0\leq\Phi_\tau(P)\leq\Phi_\tau(Q)\leq\mu^\circ(R_\tau)
 \qquad(P\subseteq Q,\ \dim P=\dim Q=n),
\]
including empty intersections.  Hence
$\Phi_\tau\in\mathcal M_n(\Sph^n)$.

By uniqueness in \autoref{lem:fan}, the extension of
$\nu|_{\Psp(\Sph^n)}$ is $\nu|_{\Ps(\Sph^n)}$.
Expanding \eqref{eq:fan} therefore yields
\begin{equation}\label{eq:residual-decomposition}
 \lambda(P)=-\sum_{j=0}^n a_jv_j^{\mathrm s}(P)
 +\sum_{\tau\neq0}(-1)^{z(\tau)}\Phi_\tau(P)
 \qquad(\dim P=n).
\end{equation}
The first term is a bounded continuous valuation.  Since $\lambda$ vanishes on
lower-dimensional spherical polytopes and $0\leq v_j^{\mathrm s}\leq1$,
\[
 \|\lambda\|_\infty
 \leq\sum_{j=0}^n|a_j|+\sum_{\tau\neq0}\|\Phi_\tau\|_\infty
 \leq\sum_{j=0}^n|a_j|+\sum_{\tau\neq0}\mu^\circ(R_\tau)<\infty.
\]
Thus \eqref{eq:residual-decomposition} verifies
\eqref{eq:regularity-class}, and $\lambda\in\cV_n$.
By \autoref{prop:vanishing}, $\lambda=0$.
Since $\overline\eta=\eta$ on proper spherical polytopes,
\[
 \mu^\circ(P)=\nu(P)+a_n v_n^{\mathrm s}(P)
 =\sum_{j=0}^n a_jv_j^{\mathrm s}(P)
 \qquad(P\in\Psp(\Sph^n)).
\]

For uniqueness, suppose that $\sum_{k=0}^n\alpha_kv_k^{\mathrm s}=0$
on $\Psp(\Sph^n)$.  By uniqueness in \autoref{lem:fan},
the same identity holds on $\Ps(\Sph^n)$.
For $0\leq j\leq n$ and a $(j+1)$-dimensional subspace $E\subseteq\R^{n+1}$,
\[
 0=\sum_{k=0}^n\alpha_kv_k^{\mathrm s}(\Sph(E))=\alpha_j.
\]
For the $U_j$-representation, \eqref{eq:U-v} and the point-value normalization imply
\[
 \mu=2\mu(\{p\})U_0+\sum_{j=0}^n a_j(U_j-U_{j+2})
 \qquad\text{on }\Psp(\Sph^n).
\]
The change of coordinates in \eqref{eq:U-v} is triangular with diagonal
entries one, so the resulting coefficients $c_0,\ldots,c_n$ are unique.
\end{proof}

\subsection{Approximation and coefficient signs}

The representing linear combination is continuous.  To extend equality
from spherical polytopes to proper closed spherical convex sets by
monotonicity, we construct approximations from both inside and outside.

For a nonempty compact convex set $L\subseteq\R^n$ and $\varepsilon>0$,
choose a finite $\varepsilon$-net $F_\varepsilon\subseteq L$ and put
$P_\varepsilon\coloneqq\operatorname{conv}F_\varepsilon$, as in
\cite[Theorem~1.8.16, p.~67]{SchneiderConvexBodies}.  Then
\[
 P_\varepsilon\subseteq L
 \subseteq F_\varepsilon+\varepsilon\overline B^n
 \subseteq P_\varepsilon+\varepsilon\overline B^n.
\]
Let $Q_\varepsilon$ be the convex hull of the vertices of all closed
coordinate grid cubes of side $\varepsilon$ meeting $L$.
Since $L$ is bounded, only finitely many such cubes meet $L$.
They cover $L$ and, being convex hulls of their vertices, lie in $Q_\varepsilon$.
For a vertex $v$ of any selected cube $C$, choose $y\in C\cap L$.
Since $\|v-y\|\leq\operatorname{diam}C=\sqrt n\,\varepsilon$ and
$L+\sqrt n\,\varepsilon\overline B^n$ is convex,
\[
 L\subseteq Q_\varepsilon\subseteq L+\sqrt n\,\varepsilon\overline B^n.
\]
Consequently, we have the inclusions and Euclidean Hausdorff bounds
\begin{equation}\label{eq:elementary-approximation}
 P_\varepsilon\subseteq L\subseteq Q_\varepsilon,
 \qquad
 d(P_\varepsilon,L)\leq\varepsilon,
 \qquad
 d(Q_\varepsilon,L)\leq\sqrt n\,\varepsilon.
\end{equation}
The construction also applies when $\dim L<n$.

For nonempty $K\in\Ksp(\Sph^n)$, choose $e\in\Sph^n$ with $K\subseteq S_e^+$ and put
$\widetilde K\coloneqq g_e(K)$.  Applying \eqref{eq:elementary-approximation}
with $\varepsilon_k=(k\max\{1,\sqrt n\})^{-1}$ gives nonempty Euclidean polytopes
$\widetilde P_k,\widetilde Q_k\subseteq e^\perp$ such that
\begin{equation}\label{eq:spherical-euclidean-approximation}
 \widetilde P_k\subseteq\widetilde K\subseteq\widetilde Q_k,
 \qquad
 \max\{d(\widetilde P_k,\widetilde K),d(\widetilde Q_k,\widetilde K)\}
  \leq k^{-1}.
\end{equation}
The proper spherical polytopes $P_k\coloneqq g_e^{-1}(\widetilde P_k)$ and
$Q_k\coloneqq g_e^{-1}(\widetilde Q_k)$ satisfy
$\varnothing\neq P_k\subseteq K\subseteq Q_k$.
Since
\[
 \widetilde Q_k\subseteq\widetilde K+k^{-1}\overline B_{e^\perp}
 \subseteq\widetilde K+\overline B_{e^\perp},
\]
the sets $\widetilde P_k,\widetilde K,\widetilde Q_k$ lie in a common
compact Euclidean ball $B\subseteq e^\perp$. Define
\[
 \omega_B(t)\coloneqq\sup\{\operatorname{dist}_{\mathrm s}
 (g_e^{-1}(z),g_e^{-1}(w)):z,w\in B,\ \|z-w\|\leq t\}.
\]
The bounds~\eqref{eq:spherical-euclidean-approximation} and uniform continuity of $g_e^{-1}|_B$
imply the spherical Hausdorff estimate
\begin{equation}\label{eq:proper-spherical-approximation}
 \max\{d_{\mathrm s}(P_k,K),d_{\mathrm s}(Q_k,K)\}
 \leq\omega_B(k^{-1})\longrightarrow0.
\end{equation}

\begin{lemma}\label{lem:monotone-approximation}
Let $\mu,\nu:\Ksp(\Sph^n)\to\R$ be functions, with $\mu$ monotone
on nonempty sets and $\nu$ continuous.  If $\mu=\nu$ on
$\Psp(\Sph^n)$, then $\mu=\nu$ on $\Ksp(\Sph^n)$.
\end{lemma}

\begin{proof}
Equality at $\varnothing$ is part of the hypothesis.
For nonempty $K$, take the proper spherical polytopes
$\varnothing\neq P_k\subseteq K\subseteq Q_k$ satisfying
\eqref{eq:proper-spherical-approximation}. Monotonicity gives
\[
 \nu(P_k)=\mu(P_k)\leq\mu(K)\leq\mu(Q_k)=\nu(Q_k).
 \]
Continuity of $\nu$ implies
 \[
 \nu(P_k)\longrightarrow\nu(K),\qquad \nu(Q_k)\longrightarrow\nu(K).
 \]
Hence $\mu(K)=\nu(K)$.
\end{proof}

The representation is now established on proper closed spherical convex sets.
It remains to determine the coefficient conditions for monotonicity.

\begin{lemma}\label{lem:coefficient-signs}
For $c_0,\ldots,c_n\in\R$, the valuation
$\nu=\sum_{j=0}^n c_jU_j$ is monotone on nonempty members of
$\Ksp(\Sph^n)$ if and only if $c_1,\ldots,c_n\geq0$.  It is monotone on all
of $\Ksp(\Sph^n)$ if and only if $c_0,\ldots,c_n\geq0$.
\end{lemma}

\begin{proof}
We first prove sufficiency on nonempty proper sets.
Each $U_j$ is monotone by \eqref{eq:Ucrofton}, and $U_0=1/2$
on nonempty sets. Thus, if
$c_1,\ldots,c_n\geq0$ and $\varnothing\neq K\subseteq L$, then
\[
 \nu(L)-\nu(K)=\sum_{j=1}^n c_j\bigl(U_j(L)-U_j(K)\bigr)\geq0.
\]

Conversely, suppose that $\nu$ is monotone on nonempty sets.
We isolate each $c_k$, $k\geq1$, by comparing nested caps that
approach hemispheres.
For an orthonormal basis $e_0,\ldots,e_n$ of $\R^{n+1}$, set
\[
 E_k=\operatorname{span}\{e_0,\ldots,e_k\},\qquad
 K_k(t)=\{x\in\Sph(E_k):\langle x,e_0\rangle\geq t\}.
\]
Here $0\leq k\leq n$ and $0\leq t<1$. For $t>0$, these are nonempty
proper spherical convex sets, since
\[
 g_{e_0}(K_k(t))
 =\{y\in E_k\cap e_0^\perp:\|y\|\leq\sqrt{t^{-2}-1}\}.
\]
Moreover, $K_0(t)=\{e_0\}$ and $K_{k-1}(t)\subseteq K_k(t)$ for
$1\leq k\leq n$.
For $k\geq1$, every $x\in K_k(0)$ has the form
$x=(\cos\theta)e_0+(\sin\theta)v$, with $0\leq\theta\leq\pi/2$ and
$v\in\Sph(E_k\cap e_0^\perp)$.
With $\theta_t=\min\{\theta,\arccos t\}$, the point
$y=(\cos\theta_t)e_0+(\sin\theta_t)v$ belongs to $K_k(t)$ and satisfies
$\operatorname{dist}_{\mathrm s}(x,y)=\theta-\theta_t\leq\arcsin t$.
Since $K_k(t)\subseteq K_k(0)$, and $K_0(t)=K_0(0)=\{e_0\}$,
\[
 d_{\mathrm s}(K_k(t),K_k(0))\leq\arcsin t\longrightarrow0.
\]

For $j\leq k$ and $L\in\Gr(\R^{n+1},n+1-j)$,
\[
 \dim(L\cap E_k)\geq k+1-j\geq1.
\]
One of the two unit directions of a line in $L\cap E_k$ lies in
$K_k(0)$. The cone $K_k(0)^\vee$ contains $e_0$ but not $-e_0$,
so it is not a linear subspace. By \eqref{eq:Ucrofton},
\[
 U_j(K_k(0))
 =\frac12\int_{\Gr(\R^{n+1},n+1-j)}1\,\dd\nu_{n+1-j}
 =\frac12\qquad(j\leq k).
\]
For $j>k$, we deduce $U_j(K_k(0))=0$ via \eqref{eq:dimension-vanishing}. Monotonicity on $K_{k-1}(t)\subseteq K_k(t)$, $t>0$, and continuity of the $U_j$ on $\Ks(\Sph^n)$ therefore yield
\[
 0\leq\lim_{t\to0^+}
 \bigl[\nu(K_k(t))-\nu(K_{k-1}(t))\bigr]
 =\frac12c_k\qquad(1\leq k\leq n).
\]
Under nonempty-set monotonicity, $\nu(K)\geq\nu(\{x\})=c_0/2$
for $x\in K\neq\varnothing$.  This lower bound is attained at points.
Since $\nu(\varnothing)=0$, comparisons with the empty set hold exactly
when $c_0\geq0$.
\end{proof}

\begin{proof}[Proof of \autoref{thm:main}\textup{(i)}]
Suppose that $\mu$ is monotone on nonempty sets.  By
\autoref{prop:polyhedral-classification}, there are unique
$c_0,\ldots,c_n\in\R$ such that
$\mu=\nu\coloneqq\sum_{j=0}^n c_jU_j$ on $\Psp(\Sph^n)$.
Each $U_j$ is continuous and $\Orth(n+1)$-invariant by \eqref{eq:Udef},
so $\nu$ has the same properties.  By \autoref{lem:monotone-approximation},
$\mu=\nu$ on $\Ksp(\Sph^n)$.
The coefficient conditions and the converse follow from
\autoref{lem:coefficient-signs}: such a sum is monotone
on nonempty sets exactly when $c_1,\ldots,c_n\geq0$, and on
$\Ksp(\Sph^n)$ exactly when $c_0,\ldots,c_n\geq0$.
\end{proof}

\subsection{Closed spherical convex sets}

The extension in \autoref{lem:fan} determines the values on nonproper
sets.  We now examine the additional monotonicity constraints imposed
by great subspheres.

\begin{lemma}\label{lem:all-closed-signs}
For $a_0,\ldots,a_n\in\R$, set
$\nu\coloneqq\sum_{j=0}^n a_jv_j^{\mathrm s}$ on $\Ks(\Sph^n)$.
The following are equivalent.
\begin{enumerate}
\renewcommand{\labelenumi}{\textup{(\roman{enumi})}}
\item $\nu$ is monotone on nonempty sets.
\item $\nu$ is monotone on $\Ks(\Sph^n)$.
\item $0\leq a_0\leq a_1\leq\cdots\leq a_n$.
\end{enumerate}
\end{lemma}

\begin{proof}
Condition~\textup{(ii)} implies~\textup{(i)} by restriction to nonempty sets.
For $\textup{(i)}\Rightarrow\textup{(iii)}$, choose nested linear subspaces
\[
 L_1\subsetneq L_2\subsetneq\cdots\subsetneq L_{n+1}=\R^{n+1},
 \qquad \dim L_k=k.
\]
For $p\in\Sph(L_1)$, the spherical intrinsic volumes of points and great subspheres give
\[
 \nu(\{p\})=\frac12a_0,\qquad
 \nu(\Sph(L_k))=a_{k-1}\quad(1\leq k\leq n+1).
\]
Monotonicity along
$\{p\}\subseteq\Sph(L_1)\subsetneq\cdots\subsetneq\Sph(L_{n+1})$ implies
\[
 \frac12a_0\leq a_0\leq a_1\leq\cdots\leq a_n,
\]
which is~\textup{(iii)}.

To prove $\textup{(iii)}\Rightarrow\textup{(ii)}$, assume~\textup{(iii)}.
Since $a_j,v_j^{\mathrm s}\geq0$,
$\nu(A)\geq0=\nu(\varnothing)$ for every nonempty $A$.
It remains to compare $\varnothing\neq A\subseteq B$.
Write $C=A^\vee$, $D=B^\vee$, and
$c_j=a_j-a_{j-2}\geq0$, where $a_{-2}=a_{-1}=0$.
If neither $C$ nor $D$ is a linear subspace, \eqref{eq:U-v} and
\eqref{eq:Ucrofton} imply
\[
 \nu(B)-\nu(A)
 =\sum_{j=0}^n c_j\bigl(U_j(B)-U_j(A)\bigr)\geq0,
\]
because each hitting indicator for $C$ is bounded by that for $D$.

If $D$ is a $k$-dimensional linear subspace, computing the intrinsic
volumes of $C\subseteq D$ within $D$ by ambient-space invariance gives
$v_r(C)=0$ for $r>k$. Nonnegativity and~\eqref{eq:conic-total-mass} yield
\[
 \nu(A)=\sum_{r=1}^k a_{r-1}v_r(C)
 \leq a_{k-1}\sum_{r=1}^k v_r(C)
 \leq a_{k-1}=\nu(B).
\]

Finally, suppose that $C$ is a $k$-dimensional linear subspace
and $D$ is not. Then $1\leq k\leq n$.
For $x\in D\setminus C$, put $u=x-\pi_Cx$.
Since $-\pi_Cx\in C\subseteq D$, we have
$u\in(D\cap C^\perp)\setminus\{0\}$, and hence
\[
 A\subseteq H\coloneqq(C+\R_{\geq0}u)\cap\Sph^n\subseteq B.
\]
The set $H$ is a closed hemisphere of $\Sph(C\oplus\R u)$.
The computation in the proof of \autoref{lem:coefficient-signs} gives
\[
 U_j(H)=
 \begin{cases}
 1/2,&0\leq j\leq k,\\
 0,&k<j\leq n.
 \end{cases}
\]
Neither $H^\vee$ nor $D$ is a linear subspace, so the first case
applies to $H\subseteq B$. Therefore,
\[
 \nu(A)=a_{k-1}
 \leq\frac{a_{k-1}+a_k}{2}
 =\frac12\sum_{j=0}^k c_j
 =\nu(H)\leq\nu(B).\qedhere
\]
\end{proof}

\begin{proof}[Proof of \autoref{thm:main}\textup{(ii)}]
Assume that $\mu$ is monotone on $\Ks(\Sph^n)$.
By \autoref{thm:main}\textup{(i)}, there are unique $c_0,\ldots,c_n\in\R$ such that the
restriction of $\mu$ to $\Ksp(\Sph^n)$ equals
$\sum_{j=0}^n c_jU_j$.  Accordingly, $\mu$ and $\sum_{j=0}^n c_jU_j$ agree on $\Ks(\Sph^n)$ by
uniqueness in \autoref{lem:fan}, since both extend this restriction.  Relation~\eqref{eq:U-v} expresses
this valuation uniquely as $\sum_{j=0}^n a_jv_j^{\mathrm s}$.  Explicitly,
\[
 a_j=\sum_{i=0}^{\lfloor j/2\rfloor}c_{j-2i},
 \qquad c_j=a_j-a_{j-2},
 \qquad 0\leq j\leq n,
\]
where $a_{-2}=a_{-1}=0$.  The inequalities $0\leq a_0\leq\cdots\leq a_n$
follow from \autoref{lem:all-closed-signs}.

Conversely, \autoref{lem:all-closed-signs} shows that every valuation with
these coefficient inequalities is monotone on $\Ks(\Sph^n)$.  Continuity and $\Orth(n+1)$-invariance follow from the
corresponding properties of the conic intrinsic volumes.
\end{proof}
\section{Conic and elliptic classifications}\label{sec:conic-elliptic}
The conic and elliptic classifications follow from the spherical results via the link map and the antipodal quotient, respectively.

\subsection{Valuations on convex cones}\label{subsec:conic-formulation}

For $n\geq1$, write $\mathcal C^{n+1}=\mathcal C(\R^{n+1})$ for all
closed convex cones and $\mathcal C_p^{n+1}$ for the pointed members.
Both families include $\{0\}$ and carry the $d_{\mathrm c}$-topology.

For $\mathcal A\in\{\mathcal C^{n+1},\mathcal C_p^{n+1}\}$, a \emph{normalized conic valuation}
is a map $\varphi:\mathcal A\to\R$ satisfying $\varphi(\{0\})=0$ and
\[
 \varphi(C)+\varphi(D)=\varphi(C\cup D)+\varphi(C\cap D),
\]
whenever $C,D,C\cup D\in\mathcal A$.
We say that $\varphi$ is \emph{monotone on nonzero cones} if
$\varphi(C)\leq\varphi(D)$ whenever
$\{0\}\neq C\subseteq D$ with $C,D\in\mathcal A$, and
simply \emph{monotone} if $\varphi(C)\leq\varphi(D)$ is required for all
$C,D\in\mathcal A$ with $C\subseteq D$.

By \autoref{lem:link-homeomorphism}, the link map is a homeomorphism
with inverse $K\mapsto K^\vee$; nonzero pointed cones correspond to proper
spherical convex sets, with $\{0\}$ corresponding to $\varnothing$.
Since $\operatorname{link}C=C\cap\Sph^n$, for $C,D,C\cup D\in\mathcal A$,
\[
 \operatorname{link}(C\cup D)=\operatorname{link}C\cup\operatorname{link}D,\qquad \operatorname{link}(C\cap D)=\operatorname{link}C\cap\operatorname{link}D.
\]
Thus the mutually inverse assignments
\[
 \mu(K)=\varphi(K^\vee),\qquad
 \varphi(C)=\mu(\operatorname{link}C)
\]
identify normalized conic valuations with spherical valuations.
They preserve continuity and the corresponding monotonicity properties.
Invariance is preserved as well, since
$\operatorname{link}(gC)=g\operatorname{link}C$ for $g\in\Orth(n+1)$.

\begin{corollary}\label{cor:conic}
Let $n\geq1$.
\begin{enumerate}
\renewcommand{\labelenumi}{\textup{(\roman{enumi})}}
\item An $\SO(n+1)$-invariant normalized conic valuation
$\varphi:\mathcal C_p^{n+1}\to\R$ is monotone
on nonzero cones if and only if it has a unique representation
\[
 \varphi=\sum_{j=0}^{n}c_jU_j,
 \qquad c_0\in\R,\quad c_1,\ldots,c_n\geq0.
\]
A valuation with this representation is monotone on $\mathcal C_p^{n+1}$
if and only if $c_0\geq0$.
\item An $\SO(n+1)$-invariant normalized conic valuation
$\varphi:\mathcal C^{n+1}\to\R$ is monotone if and only if it has a
unique representation
\[
 \varphi=\sum_{j=1}^{n+1}a_jv_j,
 \qquad 0\leq a_1\leq\cdots\leq a_{n+1}.
\]
Monotonicity of $\varphi$ on nonzero cones is equivalent to
monotonicity on $\mathcal C^{n+1}$.
\end{enumerate}
Every such valuation is continuous and $\Orth(n+1)$-invariant.
\end{corollary}

\begin{proof}
We transfer the two classifications through the link correspondence.
For a nonzero cone $C$, $(\operatorname{link}C)^\vee=C$.
Using \eqref{eq:Udef} and~\eqref{eq:spherical-intrinsic-def}, we obtain
\[
 U_j(\operatorname{link}C)=U_j(C)\;(0\leq j\leq n),\qquad v_{j-1}^{\mathrm s}(\operatorname{link}C)=v_j(C)\; (1\leq j\leq n+1).
\]
The case $C=\{0\}$ follows from $\operatorname{link}\{0\}=\varnothing$,
$v_j(\{0\})=0$ for $j\geq1$, and \eqref{eq:Udef}.
Hence \autoref{thm:main}\textup{(i)} and \textup{(ii)} transfer to the respective conic families, with the same coefficient conditions, uniqueness, continuity and orthogonal invariance. It remains to compare the two monotonicity conditions on $\mathcal C^{n+1}$.  Suppose that $\varphi$ is monotone on nonzero cones, and choose a ray $R\subseteq C\neq\{0\}$. Since $R\cup(-R)$ is a line and $R\cap(-R)=\{0\}$, the valuation identity and monotonicity imply
\[
 \varphi(\{0\})=0\leq\varphi(R\cup(-R))-\varphi(-R)=\varphi(R)\leq\varphi(C).\qedhere
\]
\end{proof}

See also \cite[Corollary~1.2 and Theorem~1.1]{LotzMonotone} for the conic
classifications.
For $n=0$, the group $\SO(1)$ is trivial and does not force equal values
on the two rays.  Thus the assumption $n\geq1$ cannot be omitted.

\subsection{Valuations on real elliptic space}\label{subsec:elliptic-corollaries}
The antipodal quotient is
\[
 q:\Sph^n\longrightarrow\mathbb{RP}^n,\qquad q(x)=[x],
\]
where $[x]$ denotes the antipodal class $\{x,-x\}$. The elliptic distance is
\[
 \operatorname{dist}_{\mathrm e}([x],[y]) =\min\{\operatorname{dist}_{\mathrm s}(x,y), \operatorname{dist}_{\mathrm s}(x,-y)\},
\]
and the isometry group is \(\Orth(n+1)/\{\pm I\}\).
We call a set \emph{elliptically convex} if it equals \(q(K)\) for
some \(K\in\Ksp(\Sph^n)\).  Thus, \(\K(\mathbb{RP}^n)\) consists of
the compact projectively convex sets contained in an affine chart,
together with \(\varnothing\).
On its nonempty members we use the
Hausdorff metric \(d_{\mathrm e}\) induced by
\(\operatorname{dist}_{\mathrm e}\). Set \(d_{\mathrm e}(\varnothing,\varnothing)=0\)
and \(d_{\mathrm e}(\varnothing,A)=d_{\mathrm e}(A,\varnothing)=\pi\)
for \(A\neq\varnothing\), so that the empty set is isolated.
A function $\nu:\Ksp(\Sph^n)\to\R$ is \emph{even} if
$\nu(-K)=\nu(K)$ for every $K\in\Ksp(\Sph^n)$.
A valuation with this symmetry is called an \emph{even valuation}.

For $A=q(K)\neq\varnothing$, the sets $K$ and $-K$ are disjoint by
properness and connected by gnomonic convexity.  Being compact, they are
the two connected components of
\[
 q^{-1}(A)=K\cup(-K).
\]
Every proper lift of $A$ is connected and maps onto $A$, so it equals
$K$ or $-K$.  More generally, if $A\subseteq q(M)$ with $M$ proper,
one of these lifts lies in $M$, and it is $M\cap q^{-1}(A)$.
Thus, for $A,B,A\cup B\in\K(\mathbb{RP}^n)$ and a lift $M$ of $A\cup B$,
the compatible lifts
\[
 K=M\cap q^{-1}(A),\qquad L=M\cap q^{-1}(B)
\]
satisfy, by injectivity of $q$ on $M$,
\[
 K\cup L=M,\qquad q(K\cap L)=A\cap B.
\]
These identities also hold for empty sets.  In particular,
$A\cap B\in\K(\mathbb{RP}^n)$.

\begin{lemma}\label{lem:elliptic-lift}
The map
\[
 \mu\longmapsto q^*\mu,\qquad
 (q^*\mu)(K)=\mu(q(K))\quad(K\in\Ksp(\Sph^n)),
\]
is a bijection from valuations on $\K(\mathbb{RP}^n)$ to even valuations
on $\Ksp(\Sph^n)$.  For $\nu=q^*\mu$,
\[
 \begin{aligned}
 \mu\text{ is continuous}
 &\Longleftrightarrow\nu\text{ is continuous},\\
 \mu\text{ is monotone on nonempty sets}
 &\Longleftrightarrow\nu\text{ is monotone on nonempty sets},\\
 \mu\text{ is monotone}
 &\Longleftrightarrow\nu\text{ is monotone},\\
 \mu\text{ is isometry-invariant}
 &\Longleftrightarrow\nu\text{ is }\Orth(n+1)\text{-invariant}.
 \end{aligned}
\]
\end{lemma}

\begin{proof}
We first establish the valuation bijection.
Since $q(-K)=q(K)$, the function $q^*\mu$ is even.
An even function $\nu$ determines a unique function $\mu$ by
$\mu(q(K))=\nu(K)$, including $\mu(\varnothing)=\nu(\varnothing)$.
This assignment is inverse to $q^*$.  Suppose that $\nu$ is a valuation.
For $A,B,A\cup B\in\K(\mathbb{RP}^n)$, choose a proper lift $M$ of
$A\cup B$ and set $K=M\cap q^{-1}(A)$, $L=M\cap q^{-1}(B)$.  Then
\[
 \mu(A)+\mu(B)
 =\nu(K)+\nu(L)
 =\nu(M)+\nu(K\cap L)
 =\mu(A\cup B)+\mu(A\cap B).
\]
Conversely, if $K\cup L$ is proper, $q$ is injective there, so
$q(K)\cap q(L)=q(K\cap L)$ and $q^*\mu$ is a valuation.
For monotonicity, take $A\subseteq B=q(L)$.  The lift $K=L\cap q^{-1}(A)$ satisfies
$K\subseteq L$.  Together with preservation of inclusions by $q$,
this proves both monotonicity equivalences.

For continuity, compare the elliptic and spherical Hausdorff distances:
for nonempty compact
$E,F\subseteq\Sph^n$,
\begin{equation}\label{eq:elliptic-hausdorff-comparison}
 d_{\mathrm e}(q(E),q(F))\leq d_{\mathrm s}(E,F).
\end{equation}
Hence $q^*\mu$ is continuous whenever $\mu$ is.
Conversely, assume that $\nu$ is continuous and
$A_r\to A=q(K)\neq\varnothing$ in $d_{\mathrm e}$.
Omit the finitely many empty terms and choose proper lifts $K_r$ of $A_r$.
The nonempty compact subsets of $\Sph^n$ form a compact space under the
spherical Hausdorff metric $d_{\mathrm s}$
\cite[Lemma~I.5.31]{BridsonHaefliger}, so every subsequence has a
further subsequence $K_{r_i}\to L$.
The limit $L$ is connected, as a Hausdorff limit of connected compact sets,
and \eqref{eq:elliptic-hausdorff-comparison} gives $q(L)=A$.  Thus
$L\subseteq K\cup(-K)$ lies in one of the two components and maps onto $A$;
consequently, $L=K$ or $L=-K$.  By continuity and evenness,
\[
 \mu(A_{r_i})=\nu(K_{r_i})\longrightarrow\nu(L)=\nu(K)=\mu(A).
\]
Thus every subsequence of $\mu(A_r)$ has a further subsequence converging to $\mu(A)$, so $\mu(A_r)\to\mu(A)$ and $\mu$ is continuous at $A$. At $\varnothing$, continuity holds on both sides because it is isolated. For invariance, use the identity
\[
 (q^*\mu)(gK)=\mu([g]q(K))\qquad(g\in\Orth(n+1)).
\]
Since every elliptic isometry is induced by such a $g$, the invariance conditions are equivalent.
\end{proof}

Define
\[
 v_j^{\mathrm e}(q(K))\coloneqq v_j^{\mathrm s}(K),\qquad
 U_j^{\mathrm e}(q(K))\coloneqq U_j(K),\qquad 0\leq j\leq n.
\]
The definitions are independent of the lift by evenness of $v_j^{\mathrm s}$
and $U_j$.  By \autoref{lem:elliptic-lift}, these are continuous
isometry-invariant valuations on $\K(\mathbb{RP}^n)$.
Only the classifications on proper spherical convex sets apply here.
The great subspheres used to compare adjacent coefficients in
\autoref{lem:all-closed-signs} are not proper and hence are excluded
from the lifting correspondence.

\begin{corollary}\label{cor:elliptic-continuous}\label{cor:elliptic-monotone}
Let $n\geq1$ and $\mu:\K(\mathbb{RP}^n)\to\R$ be an
isometry-invariant valuation.
\begin{enumerate}
\renewcommand{\labelenumi}{\textup{(\roman{enumi})}}
\item $\mu$ is continuous if and only if it has a unique representation
\[
 \mu=\sum_{j=0}^n a_jv_j^{\mathrm e},\qquad a_0,\ldots,a_n\in\R.
\]
\item $\mu$ is monotone on nonempty sets if and only if it has a
unique representation
\[
 \mu=\sum_{j=0}^n c_jU_j^{\mathrm e},\qquad
 c_0\in\R,\quad c_1,\ldots,c_n\geq0.
\]
Every valuation with this representation is continuous.
It is monotone on $\K(\mathbb{RP}^n)$ if and only if $c_0\geq0$.
\end{enumerate}
\end{corollary}

\begin{proof}
By \autoref{lem:elliptic-lift}, $\nu=q^*\mu$ is an $\Orth(n+1)$-invariant
valuation on $\Ksp(\Sph^n)$ with the same continuity and monotonicity
properties as $\mu$.  The pullback identities
$q^*v_j^{\mathrm e}=v_j^{\mathrm s}$ and $q^*U_j^{\mathrm e}=U_j$
imply, for real coefficients $a_j,c_j$,
\[
 \nu=\sum_{j=0}^n a_jv_j^{\mathrm s}\Longleftrightarrow\mu=\sum_{j=0}^n a_jv_j^{\mathrm e},\quad \nu=\sum_{j=0}^n c_jU_j\Longleftrightarrow\mu=\sum_{j=0}^n c_jU_j^{\mathrm e},
\]
since every elliptic convex set has the form $q(K)$.
For \textup{(i)}, apply Wang and Wu's continuous classification on
$\Ksp(\Sph^n)$ \cite[Corollary~8.3]{WangWu} to the first equivalence.
For \textup{(ii)}, apply \autoref{thm:main}\textup{(i)} to the second.
The respective spherical classifications give uniqueness of the coefficients,
continuity in \textup{(ii)}, and the additional condition $c_0\geq0$
for monotonicity including $\varnothing$.
\end{proof}

The elliptic classification in \autoref{cor:elliptic-monotone}\textup{(ii)} follows from \autoref{thm:main}\textup{(i)}, since elliptic convex sets lift to proper spherical convex sets. The coefficient ordering in \autoref{thm:main}\textup{(ii)} is not necessary in elliptic space, as the following computation and Euler example show. Set $U_{n+1}^{\mathrm e}=U_{n+2}^{\mathrm e}=0$. By \eqref{eq:U-v}, $v_j^{\mathrm e}=U_j^{\mathrm e}-U_{j+2}^{\mathrm e}$ for $0\leq j\leq n$.  Thus, for $\mu=\sum_{j=0}^n a_jv_j^{\mathrm e}$ and $a_{-2}=a_{-1}=0$,
\begin{equation}\label{eq:elliptic-basis-change}
 \mu=\sum_{j=0}^n a_j\bigl(U_j^{\mathrm e}-U_{j+2}^{\mathrm e}\bigr) =\sum_{j=0}^n(a_j-a_{j-2})U_j^{\mathrm e}.
\end{equation}
Consequently, \autoref{cor:elliptic-monotone}\textup{(ii)} gives
\[
 \mu\text{ is monotone on nonempty sets} \quad\Longleftrightarrow\quad a_j-a_{j-2}\geq0\quad(1\leq j\leq n).
\]
Thus the even and odd coefficients form separate nondecreasing sequences with $a_1\geq0$ and $a_0$ unrestricted; monotonicity including $\varnothing$ additionally requires $a_0\geq0$. For example, the Euler valuation
\[
 2U_0^{\mathrm e}(A)=2\sum_{k=0}^{\lfloor n/2\rfloor}v_{2k}^{\mathrm e}(A)=\begin{cases}1,&A\neq\varnothing,\\0,&A=\varnothing,\end{cases}
\]
is monotone, but its coefficients satisfy $a_{2k}=2$ and $a_{2k+1}=0$, so $a_0>a_1$.

\section{Hyperbolic geometry and Crofton valuations}
\label{sec:hyperbolic}
We now turn to hyperbolic convexity, working throughout the remaining sections in $\mathbb H^n$ with $n\geq1$ and curvature $-1$.  This section establishes approximation in the Klein model, and the basic properties and restriction formulas of Crofton valuations.
\subsection{The Klein model and approximation}\label{subsec:hb-klein}
Denote by $\K(\mathbb H^n)$ the family of compact geodesically convex
sets, including $\varnothing$, and by $\mathcal P(\mathbb H^n)$
the family consisting of $\varnothing$ and the finite geodesic convex hulls
\[
 \operatorname{conv}_{\mathbb H}\{p_1,\ldots,p_m\},
 \qquad p_1,\ldots,p_m\in\mathbb H^n,\quad m\geq1.
\]
Members of $\mathcal P(\mathbb H^n)$ are called \emph{hyperbolic polytopes}.
Valuations are real-valued and vanish at $\varnothing$.
A valuation on $\mathcal P(\mathbb H^n)$ is \emph{simple} if it vanishes
on every hyperbolic polytope of dimension less than $n$.
Invariance always refers to the full group
$\operatorname{Isom}(\mathbb H^n)$, including reflections.

We denote hyperbolic distance by $\operatorname{dist}_{\mathbb H}$ and
hyperbolic Riemannian volume by $\operatorname{vol}_n$
\cite{doCarmoRiemannian}.  For Euclidean volume, we use
$\operatorname{vol}^{\mathrm E}_k$ for $k$-dimensional Lebesgue measure.
We equip the nonempty members of $\K(\mathbb H^n)$ with the induced
Hausdorff metric $d_{\mathbb H}$ and adjoin $\varnothing$ as an isolated
point.  A compact hyperbolic ball means a closed ball of finite radius
for $\operatorname{dist}_{\mathbb H}$.

In the Klein model $\mathbb H^n=B^n$, geodesic segments are Euclidean
segments and complete totally geodesic $k$-subspaces have the form
$B^n\cap A$, where $A$ is an affine $k$-plane
\cite[Chapter~I.6]{BridsonHaefliger}.
Thus the nonempty members of $\mathcal P(\mathbb H^n)$ are precisely
the Euclidean polytopes contained in $B^n$.
Orthogonal projections and complements are Euclidean.
The hyperbolic metric tensor and volume density are \cite[Equation~(5.1), Lemma~5.1, and Remark~5.2]{BetkenHugThale}
\begin{equation}\label{eq:hb-klein-metric-density}
 (g_{ij}(x))=\frac{I}{1-\|x\|^2}
       +\frac{xx^{\mathsf T}}{(1-\|x\|^2)^2},\qquad
 \dd\operatorname{vol}_n(x)=J(x)\dd x,\quad
  J(x)=(1-\|x\|^2)^{-(n+1)/2}.
\end{equation}

Compactness supplies a uniform bound for continuous valuations on sets
contained in a fixed compact hyperbolic ball.  Two-sided approximation extends the
representation on hyperbolic polytopes in the monotone case.

\begin{lemma}\label{lem:hb-klein-compactness}
Let $B\subseteq\mathbb H^n$ be a compact hyperbolic ball.
The family of nonempty members of $\K(\mathbb H^n)$ contained in $B$
is compact with respect to $d_{\mathbb H}$.
For every full-dimensional $K\in\K(\mathbb H^n)$, there are hyperbolic polytopes
$R_i^-,R_i^+\in\mathcal P(\mathbb H^n)$ in a common compact hyperbolic ball such that
\[
 R_i^-\subseteq K\subseteq R_i^+,\qquad
 \max\{d_{\mathbb H}(R_i^-,K),d_{\mathbb H}(R_i^+,K)\}\longrightarrow0.
\]
\end{lemma}

\begin{proof}
We first prove $d_{\mathbb H}$-compactness by comparison with the
Euclidean Hausdorff distance $d$.
For $0<r<1$, \eqref{eq:hb-klein-metric-density} and the Cauchy--Schwarz inequality imply
\[
 \|v\|^2\leq v^{\mathsf T}(g_{ij}(x))v
 \leq(1-r^2)^{-2}\|v\|^2
 \qquad(x\in r\overline B^n,\ v\in\R^n).
\]
For $a,b\in r\overline B^n$, the segment $[a,b]\subseteq r\overline B^n$
is a minimizing hyperbolic geodesic.  With $x_t=(1-t)a+tb$, its length satisfies
\[
 \|a-b\|\leq\operatorname{dist}_{\mathbb H}(a,b)
 =\int_0^1\sqrt{(b-a)^{\mathsf T}(g_{ij}(x_t))(b-a)}\dd t
 \leq(1-r^2)^{-1}\|a-b\|.
\]
Taking the infimum over one set and the supremum over the other,
in both directions, yields
\begin{equation}\label{eq:hb-klein-metric-comparison}
 d(A,D)\leq d_{\mathbb H}(A,D)\leq(1-r^2)^{-1}d(A,D)
 \qquad(\varnothing\neq A,D\subseteq r\overline B^n\text{ compact}).
\end{equation}
Choose $0<r<1$ with $B\subseteq r\overline B^n$.
By the Blaschke selection theorem
\cite[Theorem~1.8.7]{SchneiderConvexBodies}, any sequence of nonempty
compact Euclidean convex sets $K_i\subseteq B$ has a subsequence $K_{i_j}$ converging in $d$
to a nonempty compact Euclidean convex set $K$.  For $x\in K$,
\[
 \inf_{y\in B}\|x-y\|\leq d(K,K_{i_j})\longrightarrow0.
\]
Since $B$ is closed, $K\subseteq B$; then
\eqref{eq:hb-klein-metric-comparison} implies
$d_{\mathbb H}(K_{i_j},K)\to0$, proving compactness.

For the approximation, choose $0<r<1$ with $K\subseteq rB^n$.
Apply \eqref{eq:elementary-approximation} with
$\varepsilon=(i\sqrt n)^{-1}$ to obtain Euclidean polytopes satisfying
\[
 R_i^-\subseteq K\subseteq R_i^+,\qquad
 \max\{d(R_i^-,K),d(R_i^+,K)\}\leq i^{-1}.
\]
For all sufficiently large $i$,
\[
 \max_{x\in R_i^+}\|x\|\leq\max_{x\in K}\|x\|+i^{-1}<r.
\]
Thus $R_i^-$ and $R_i^+$ lie in $r\overline B^n$, the compact hyperbolic ball
of radius $\operatorname{arctanh}r$ about $0$, and are hyperbolic
polytopes because convex hulls agree in the Klein model.
After discarding finitely many terms,
\eqref{eq:hb-klein-metric-comparison} completes the proof:
\[
 \max\{d_{\mathbb H}(R_i^-,K),d_{\mathbb H}(R_i^+,K)\}
 \leq(1-r^2)^{-1}i^{-1}\longrightarrow0.\qedhere
\]
\end{proof}

\subsection{Crofton valuations and restriction}\label{subsec:hb-crofton}

Set \(\chi(K)\coloneqq1\) for every nonempty
\(K\in\K(\mathbb H^n)\), and set \(\chi(\varnothing)\coloneqq0\).
For \(1\leq k<n\), \(\operatorname{AGr}_k(\mathbb H^n)\) denotes the
space of complete \(k\)-dimensional totally geodesic subspaces of
\(\mathbb H^n\).  Every \(L\in\operatorname{AGr}_k(\mathbb H^n)\) has a
unique Klein
representation
\[
 (u+L_0)\cap B^n,
 \qquad L_0\in\Gr(\R^n,k),\quad u\in L_0^\perp,\quad \|u\|<1.
\]
Write \(\nu_k\) for the invariant probability measure on
\(\Gr(\R^n,k)\), and \(\dd u\) for Euclidean Lebesgue measure on
\(L_0^\perp\).  This space of totally geodesic subspaces admits a
nonzero isometry-invariant Radon measure
\cite[Section~17.3, pp.~304--306]{Santalo}.  Since the isometry group
acts transitively, this measure is unique up to a positive factor
\cite[Theorem~13.3.1]{SchneiderWeil}.  We normalize it, denoting it by
\(m_k^{(n)}\), so that
the associated quermassintegrals have the Andrews--Wei normalization
\cite[Equations~(1.3)--(1.6)]{AndrewsWei}.
In these coordinates, the measure takes the form \cite[Lemma~5.3]{BetkenHugThale}:
\begin{equation}\label{eq:hb-plane-density}
 \dd m_k^{(n)}(L)=\rho_k(L_0,u)\dd u\,\dd\nu_k(L_0),\qquad
 \rho_k(L_0,u)=\beta_{n,k}(1-\|u\|^2)^{-(n+1)/2},\quad\beta_{n,k}>0.
\end{equation}
By \cite[Equation~(6)]{HeroldHugThale}, this measure is proportional to
\[
 \cosh^k\operatorname{dist}_{\mathbb H}(0,u)\,
 \dd\operatorname{vol}_{n-k}(u)\,\dd\nu_k(L_0).
\]
In the Klein model, \cite[Theorem~6.1.1]{Ratcliffe} and
\eqref{eq:hb-klein-metric-density} give the two density factors
$(1-\|u\|^2)^{-k/2}$ and
$(1-\|u\|^2)^{-(n-k+1)/2}$, respectively.
Their product is the density in \eqref{eq:hb-plane-density}, which is
positive and smooth for $\|u\|<1$.
Define
\(W_0^{(n)}\coloneqq\operatorname{vol}_n\) and, for \(1\leq k<n\),
\begin{equation}\label{eq:hb-crofton}
 W_k^{(n)}(K)
 \coloneqq\int_{\operatorname{AGr}_k(\mathbb H^n)}
   \one_{\{L\cap K\neq\varnothing\}}\,\dd m_k^{(n)}(L),
 \qquad K\in\K(\mathbb H^n),
\end{equation}
where \(W_k^{(n)}(\varnothing)\coloneqq0\).  With this scale,
\[
 W_1^{(n)}(K)=\frac1n\operatorname{area}_{n-1}(\partial K)
 \quad(n\geq2)
\]
for every full-dimensional \(K\).  This normalization agrees with
Andrews and Wei after replacing their
ambient dimension \(n+1\) by \(n\).
The Crofton formula \eqref{eq:hb-crofton} also defines \(W_k^{(n)}\) on lower-dimensional
compact hyperbolic convex sets.

\begin{lemma}\label{lem:hb-crofton-properties}
The valuations $\chi,W_0^{(n)},\ldots,W_{n-1}^{(n)}$ on
$\K(\mathbb H^n)$ are continuous, isometry-invariant and monotone,
and their restrictions to $\mathcal P(\mathbb H^n)$ are linearly independent.
Moreover, $W_k^{(n)}(\{x\})=0$ for every $x\in\mathbb H^n$ and $0\leq k<n$.
For $n\geq2$ and every complete totally geodesic hyperplane
$H\cong\mathbb H^{n-1}$,
\begin{equation}\label{eq:hb-restriction}
 \left.W_0^{(n)}\right|_{\K(H)}=0,\qquad
 \left.W_k^{(n)}\right|_{\K(H)}
 =\alpha_{n,k}W_{k-1}^{(n-1)}\qquad(1\leq k<n),
\end{equation}
where $\alpha_{n,k}>0$ depends only on $n$ and $k$.
\end{lemma}

\begin{proof}
We begin with the valuation, invariance and monotonicity assertions.
Regard points as $0$-planes and set $m_0^{(n)}=\operatorname{vol}_n$
in each dimension.  Then \eqref{eq:hb-plane-density} and
\eqref{eq:hb-crofton} also hold for $k=0$, with $E=\{0\}$,
$\nu_0$ the unit mass at $\{0\}$, and $\beta_{n,0}=1$.
For $0\leq k<n$ and $L=(u+E)\cap B^n$, put
\[
 h_K(L)\coloneqq\one_{\{L\cap K\neq\varnothing\}}
 =\one_{\pi_{E^\perp}K}(u)
\]
and $h_\varnothing=0$.  For $K\neq\varnothing$, joint Borel measurability
follows from continuity of
$(E,u)\mapsto\min_{x\in K}\|u-\pi_{E^\perp}x\|$.
The family $\mathcal A_r=\{L:L\cap r\overline B^n\neq\varnothing\}$
is compact for $0<r<1$: its parameters satisfy
$E\in\Gr(\R^n,k)$, $u\in E^\perp$ and $\|u\|\leq r$.
Hence, for $K\subseteq r\overline B^n$,
\[
 0\leq W_k^{(n)}(K)\leq m_k^{(n)}(\mathcal A_r)<\infty.
\]
If $K\cup K'$ is convex, the closed sets $K\cap L$ and $K'\cap L$
have connected union and therefore intersect whenever both are nonempty.
Consequently,
\[
 h_K+h_{K'}=h_{K\cup K'}+h_{K\cap K'}.
\]
Integration proves that $W_k^{(n)}$ is a valuation.  Invariance of
$m_k^{(n)}$ gives, for every isometry $g$,
\[
 W_k^{(n)}(gK)
 =\int_{\operatorname{AGr}_k(\mathbb H^n)}h_{gK}(gL)\dd m_k^{(n)}(L)
 =W_k^{(n)}(K).
\]
Monotonicity follows by integrating $h_K\leq h_{K'}$ for $K\subseteq K'$.

For continuity, suppose that $K_i\to K$ in $d_{\mathbb H}$,
with all sets nonempty.
All these sets lie in $r\overline B^n$ for some $r<1$ after discarding
finitely many terms.  For fixed $E$, write
$D_i=\pi_{E^\perp}K_i$ and $D=\pi_{E^\perp}K$.  By
\eqref{eq:hb-klein-metric-comparison},
\[
 \delta_i\coloneqq d(D_i,D)\leq d(K_i,K)\leq d_{\mathbb H}(K_i,K)\to0.
\]
If $u\notin D$, then $\operatorname{dist}(u,D)>\delta_i$ eventually,
so $u\notin D_i$.  If $u+\varepsilon\overline B_{E^\perp}\subseteq D$
for some $\varepsilon>0$, then, for $\delta_i<\varepsilon$ and every unit
$v\in E^\perp$,
\[
 \max_{w\in D_i}\langle w,v\rangle
 \geq\max_{w\in D}\langle w,v\rangle-\delta_i
 \geq\langle u,v\rangle+\varepsilon-\delta_i
 >\langle u,v\rangle.
\]
Separation therefore implies $u\in D_i$.
Thus $\one_{D_i}\to\one_D$ off $\partial D$, with boundary taken in
$E^\perp$.  This boundary is Lebesgue-null: if $D$ has empty interior,
it lies in a proper affine subspace; otherwise, for $a\in\operatorname{int}D$
and $0<t<1$, the inclusion $a+t(D-a)\subseteq\operatorname{int}D$ gives
\[
 \operatorname{vol}^{\mathrm E}_{n-k}(\partial D)
 \leq(1-t^{n-k})\operatorname{vol}^{\mathrm E}_{n-k}(D)
 \longrightarrow0\qquad(t\uparrow1).
\]
Fubini and \eqref{eq:hb-plane-density} now imply $h_{K_i}\to h_K$
almost everywhere for $m_k^{(n)}$.
Since $|h_{K_i}-h_K|\leq\one_{\mathcal A_r}$ and
$m_k^{(n)}(\mathcal A_r)<\infty$, dominated convergence yields
\[
 |W_k^{(n)}(K_i)-W_k^{(n)}(K)|
 \leq\int_{\operatorname{AGr}_k(\mathbb H^n)}
       |h_{K_i}(L)-h_K(L)|\dd m_k^{(n)}(L)\longrightarrow0.
\]
Continuity at $\varnothing$ follows because it is isolated.

For $K=\{x\}$, every projection $\pi_{E^\perp}K$ is a singleton in the positive-dimensional space $E^\perp$, so \eqref{eq:hb-plane-density} gives $W_k^{(n)}(\{x\})=0$. For the Euler valuation $\chi$, consider nonempty $K,K'$ with convex union. Connectedness implies $K\cap K'\neq\varnothing$, so the valuation identity for $\chi$ reads $1+1=1+1$; if either set is empty, the identity is immediate.
The values $\chi(K)=1$ for $K\neq\varnothing$ and $\chi(\varnothing)=0$, with $\varnothing$ isolated, also prove invariance, monotonicity and continuity.

For the restriction formula \eqref{eq:hb-restriction}, fix $1\leq k<n$. Invariance reduces us to $H=B^n\cap e_n^\perp$.  The exceptional directions satisfy
\begin{equation}\label{eq:hb-exceptional-directions}
 \nu_k\{E\in\Gr(\R^n,k):E\subseteq e_n^\perp\}=0.
\end{equation}
Indeed, the span of $k$ independent standard Gaussian vectors has law
$\nu_k$, and containment in this hyperplane requires every vector to
have zero last coordinate.
By \eqref{eq:hb-plane-density}, the planes contained in $H$ are
therefore $m_k^{(n)}$-null.
Define
\[
 \mathcal T_H=\{L:L\cap H\neq\varnothing,\ L\not\subseteq H\},\qquad
 r_H:\mathcal T_H\to\operatorname{AGr}_{k-1}(H),\quad L\mapsto L\cap H.
\]
For $v=\pi_Ee_n\neq0$, the affine intersection $(u+E)\cap e_n^\perp$
has direction $F=E\cap e_n^\perp$ and orthogonal offset $z$, where
\begin{equation}\label{eq:hb-intersection-coordinates}
 z=u-\frac{\langle u,e_n\rangle}{\|v\|^2}v,\qquad
  \pi_F=\pi_E-\frac{vv^{\mathsf T}}{\|v\|^2}.
\end{equation}
Indeed, $z\in(u+E)\cap e_n^\perp$, $z\perp F$ and $E=F\oplus\R v$.
Thus $\mathcal T_H$ is the open set specified by $v\neq0$ and $\|z\|<1$,
and \eqref{eq:hb-intersection-coordinates} shows that $r_H(L)=(z+F)\cap B^n$ is continuous.
Taking $u=0$ and $e_n\in E$ shows that $\mathcal T_H\neq\varnothing$.

The positive density in \eqref{eq:hb-plane-density} therefore makes
$\eta=(r_H)_*(m_k^{(n)}|_{\mathcal T_H})$ a nonzero Borel measure.
For a nonempty compact family $C\subseteq\operatorname{AGr}_{k-1}(H)$,
the norm of the orthogonal offset is continuous and attains a maximum
less than $1$.  Thus, for some $0<r<1$, every member of $C$ meets the
compact hyperbolic ball $B_r=H\cap r\overline B^n$ in $H$, and
\[
 \eta(C)\leq m_k^{(n)}\{L:L\cap B_r\neq\varnothing\}
          =W_k^{(n)}(B_r)<\infty.
\]
Hence $\eta$ is Radon on the locally compact, second-countable space
$\operatorname{AGr}_{k-1}(H)$.

In the hyperboloid model, an isometry $h$ of $H$ is a Lorentz transformation \cite[Theorem~3.2.3]{Ratcliffe};
extending it by the identity on the normal line gives an ambient isometry $g$.
Since $r_H(gL)=h\,r_H(L)$, for every Borel $A\subseteq\operatorname{AGr}_{k-1}(H)$,
\[
 \eta(hA)=m_k^{(n)}(r_H^{-1}(hA))
 =m_k^{(n)}(g\,r_H^{-1}(A))=\eta(A).
\]
Since $\operatorname{Isom}(H)$ acts transitively on
$\operatorname{AGr}_{k-1}(H)$, uniqueness of invariant Radon measures
\cite[Theorem~13.3.1]{SchneiderWeil} implies
$\eta=\alpha_{n,k}m_{k-1}^{(n-1)}$ for some $\alpha_{n,k}>0$.
For $K\in\K(H)$, planes outside $\mathcal T_H$ either miss $K$ or lie in
the null family of planes contained in $H$.  Therefore
\[
 W_k^{(n)}(K)
 =\int_{\operatorname{AGr}_{k-1}(H)}
       \one_{\{M\cap K\neq\varnothing\}}\dd\eta(M)
 =\alpha_{n,k}W_{k-1}^{(n-1)}(K).
\]
Also, $W_0^{(n)}(K)=0$ on $H$.  Ambient isometries between complete totally geodesic hyperplanes
show that $\alpha_{n,k}$ is independent of $H$.

For linear independence, suppose that $a\chi+\sum_{k=0}^{n-1}b_kW_k^{(n)}=0$ on
$\mathcal P(\mathbb H^n)$.
At a point this identity reduces to $a=0$.
In dimension $1$, a nondegenerate hyperbolic segment $I$ gives
$b_0\operatorname{vol}_1(I)=0$, hence $b_0=0$.
For $n\geq2$, restriction to a complete totally geodesic hyperplane and induction using
\eqref{eq:hb-restriction} yield
\[
 \sum_{k=1}^{n-1}\alpha_{n,k}b_kW_{k-1}^{(n-1)}=0
 \quad\Longrightarrow\quad b_1=\cdots=b_{n-1}=0.
\]
Evaluation on a full-dimensional hyperbolic polytope, whose hyperbolic volume is positive, forces $b_0=0$.
\end{proof}

\section{Simple hyperbolic valuations}
\label{sec:hyperbolic-simple}
We first establish finite cutting, local boundedness and apex measurability, then use reflection averaging to identify simple invariant valuations with multiples of hyperbolic volume.

\subsection{Cutting and apex measurability}

For a simple valuation $\lambda:\mathcal P(\mathbb H^n)\to\R$,
$Q\in\mathcal P(\mathbb H^n)$, and a complete totally geodesic
hyperplane $H$ with closed halfspaces $H^\pm$, the valuation identity reduces to
\begin{equation}\label{eq:hb-single-cut}
 \lambda(Q)=\lambda(Q\cap H^+)+\lambda(Q\cap H^-),
\end{equation}
since $Q\cap H^\pm$ are hyperbolic polytopes and
$\lambda(Q\cap H)=0$ by simplicity.

\begin{lemma}\label{lem:hb-finite-cuts}
Let $\lambda:\mathcal P(\mathbb H^n)\to\R$ be a simple valuation, $Q\in\mathcal P(\mathbb H^n)$, and $H_1,\ldots,H_q$ complete totally geodesic hyperplanes, where $q\geq0$. For the closed halfspaces $H_i^\pm$ bounded by $H_i$, set $Q_\varepsilon=Q\cap\bigcap_{i=1}^qH_i^{\varepsilon_i}$, where $\varepsilon\in\{+,-\}^q$. Then
\[
 \lambda(Q)=
 \sum_{\varepsilon\in\{+,-\}^q,\,\dim Q_\varepsilon=n}
 \lambda(Q_\varepsilon).
\]
\end{lemma}

\begin{proof}
Iterating~\eqref{eq:hb-single-cut} over $H_1,\ldots,H_q$ yields
\[
 \lambda(Q)
 =\sum_{\varepsilon\in\{+,-\}^q}\lambda(Q_\varepsilon)
 =\sum_{\varepsilon\in\{+,-\}^q,\,\dim Q_\varepsilon=n}
       \lambda(Q_\varepsilon),
\]
where the last equality uses simplicity and $\lambda(\varnothing)=0$.
For $q=0$, the only sign tuple is empty and $Q_\varepsilon=Q$,
so the same equalities hold.
\end{proof}

We do not assume monotonicity of the simple residual.  Instead, the
following lemma derives its local boundedness and apex measurability
from the original valuation and the continuous term subtracted from it.
Write $[Q,x]=\operatorname{conv}_{\mathbb H}(Q\cup\{x\})$.

\begin{lemma}\label{lem:hb-apex}
Let $\mu,\nu:\K(\mathbb H^n)\to\R$ be valuations.
Assume that $\mu$ is isometry-invariant and either continuous or
monotone on nonempty sets, that $\nu$ is continuous, and that
$\lambda=\mu-\nu$ is simple on hyperbolic polytopes.
Then, for every compact hyperbolic ball $B\subseteq\mathbb H^n$,
\begin{equation}\label{eq:hb-local-bound}
 \sup_{R\in\mathcal P(\mathbb H^n),\,R\subseteq B}
 |\lambda(R)|<\infty.
\end{equation}
For every nondegenerate hyperbolic $(n-1)$-simplex $Q$, the function
$x\mapsto\lambda([Q,x])$ is measurable with respect to completed hyperbolic volume.
\end{lemma}

\begin{proof}
For the local bound \eqref{eq:hb-local-bound}, fix a compact hyperbolic ball $B$.  The family
$\mathcal K_B=\{\varnothing\neq K\in\K(\mathbb H^n):K\subseteq B\}$
is compact by \autoref{lem:hb-klein-compactness}.
If $\mu$ is continuous, then $\lambda=\mu-\nu$ is continuous and bounded
on $\mathcal K_B$.
If $\mu$ is monotone on nonempty sets, isometry invariance makes
$a=\mu(\{p\})$ independent of $p$, and
\[
 a\leq\mu(R)\leq\mu(B),\qquad
 |\lambda(R)|\leq\max\{|a|,|\mu(B)|\}
       +\sup_{K\in\mathcal K_B}|\nu(K)|<\infty
 \quad(R\in\mathcal K_B).
\]
The supremum is finite by continuity of $\nu$ on $\mathcal K_B$.
Since $\lambda(\varnothing)=0$, both cases prove \eqref{eq:hb-local-bound}.

For apex measurability, fix a nondegenerate hyperbolic $(n-1)$-simplex $Q$.
Klein convex hulls satisfy
\[
 d([Q,x],[Q,y])\leq\|x-y\|:
\]
the point $(1-s)z+sx$, with $z\in Q$ and $0\leq s\leq1$, is within
Euclidean distance $s\|x-y\|$ of $(1-s)z+sy$, and conversely.
Together with \eqref{eq:hb-klein-metric-comparison}, this proves continuity
of $x\mapsto[Q,x]$ in $d_{\mathbb H}$.
Hence $x\mapsto\nu([Q,x])$ is continuous; if $\mu$ is continuous,
$x\mapsto\lambda([Q,x])$ is continuous as well.

It remains to prove measurability when $\mu$ is monotone.
Choose Klein coordinates in which $Q=Q'\times\{0\}$ and
$H=B^n\cap\{t=0\}$ contains $Q$.
For either $\varepsilon\in\{-1,1\}$, write
$\Omega_\varepsilon=B^n\cap\{\varepsilon t>0\}$.  The map
\[
 \phi_\varepsilon:D\to\Omega_\varepsilon,\qquad
 \phi_\varepsilon(v,u)=(v/u,\varepsilon/u),\qquad
 D=\{(v,u)\in\R^{n-1}\times\R:u>\sqrt{1+\|v\|^2}\},
\]
is a diffeomorphism with inverse
$(z,t)\mapsto(\varepsilon z/t,\varepsilon/t)$.
The convex cone $C_Q=\{s(y,1):s\geq0,\ y\in Q'\}$ contains the nonempty
open set $\{(v,u):u>0,\ v/u\in\operatorname{int}Q'\}$.
Put $f_\varepsilon(w)=\mu([Q,\phi_\varepsilon(w)])$.
For $w=(v,u)$ and $w'=w+s(y,1)$ in $D$, with $s\geq0$ and $y\in Q'$,
\[
 \phi_\varepsilon(w')
 =\frac{u}{u+s}\phi_\varepsilon(w)+\frac{s}{u+s}(y,0).
\]
Since these coefficients are nonnegative and sum to $1$,
\[
 [Q,\phi_\varepsilon(w')]\subseteq[Q,\phi_\varepsilon(w)],\qquad
 f_\varepsilon(w')\leq f_\varepsilon(w).
\]
Thus $-f_\varepsilon$ is $C_Q$-monotone on the open set $D$,
and \autoref{lem:bbl} implies that $f_\varepsilon$ is Lebesgue measurable.

To transfer this measurability to $\Omega_\varepsilon$, compute
\[
 D\phi_\varepsilon(v,u)
 =\begin{pmatrix}u^{-1}I_{n-1}&-u^{-2}v\\0&-\varepsilon u^{-2}\end{pmatrix},
 \qquad |\det D\phi_\varepsilon(v,u)|=u^{-(n+1)}.
\]
Since $1-\|\phi_\varepsilon(v,u)\|^2=(u^2-\|v\|^2-1)/u^2$,
change of variables gives, for every Borel $A\subseteq D$,
\[
 \begin{aligned}
 \operatorname{vol}_n(\phi_\varepsilon(A))
 &=\int_A J(\phi_\varepsilon(v,u))
           |\det D\phi_\varepsilon(v,u)|\dd v\,\dd u\\
 &=\int_A(u^2-\|v\|^2-1)^{-(n+1)/2}\dd v\,\dd u.
 \end{aligned}
\]
The pullback measure $A\mapsto\operatorname{vol}_n(\phi_\varepsilon(A))$
therefore has strictly positive density on $D$.  By
\eqref{eq:positive-density-nullsets}, its completion has the same
measurable sets as Lebesgue measure.
Since $\phi_\varepsilon$ is a homeomorphism, for every $c\in\R$ the identity
\[
 \phi_\varepsilon^{-1}
 \{x\in\Omega_\varepsilon:\mu([Q,x])>c\}
 =\{w\in D:f_\varepsilon(w)>c\}
\]
proves measurability of $x\mapsto\mu([Q,x])$ for completed hyperbolic
volume on $\Omega_\varepsilon$.
Subtract $x\mapsto\nu([Q,x])$, already proved continuous.
The two choices of $\varepsilon$ cover $\mathbb H^n\setminus H$,
and $\lambda([Q,x])=0$ on $H$ by simplicity, completing the proof.
\end{proof}

\subsection{Simple valuations and hyperbolic volume}

Fix an orientation of $\mathbb H^n$.  For a nondegenerate ordered hyperbolic simplex
$[x_0,\ldots,x_n]$, write $\varepsilon(x_0,\ldots,x_n)$ for its orientation
sign.  The signed simplex function of a simple valuation $\lambda$ is
\begin{equation}\label{eq:hb-simplex-cocycle}
 f(x_0,\ldots,x_n)=
 \begin{cases}
  \varepsilon(x_0,\ldots,x_n)\lambda([x_0,\ldots,x_n]),
     &\dim[x_0,\ldots,x_n]=n,\\
  0,&\dim[x_0,\ldots,x_n]<n.
 \end{cases}
\end{equation}

\begin{lemma}\label{lem:hb-cocycle}
For a simple valuation $\lambda$ on $\mathcal P(\mathbb H^n)$, the
function $f$ in \eqref{eq:hb-simplex-cocycle} is alternating and satisfies
\begin{equation}\label{eq:hb-cocycle-identity}
 \sum_{i=0}^{n+1}(-1)^i
 f(x_0,\ldots,\widehat{x_i},\ldots,x_{n+1})=0
\end{equation}
for all $x_0,\ldots,x_{n+1}\in\mathbb H^n$.  If $\lambda$ is
isometry-invariant, then
\begin{equation}\label{eq:hb-sign-equivariance}
 f(gx_0,\ldots,gx_n)=\sgn(g)f(x_0,\ldots,x_n)
\end{equation}
for every isometry $g$, where $\sgn(g)=1$ if $g$ preserves orientation
and $\sgn(g)=-1$ otherwise.
\end{lemma}

\begin{proof}
Permuting the vertices preserves the simplex and multiplies its orientation
by the sign of the permutation, so $f$ is alternating.
If $\lambda$ is isometry-invariant, \eqref{eq:hb-sign-equivariance} holds
because $g$ multiplies the orientation by $\sgn(g)$.
To prove the cocycle identity \eqref{eq:hb-cocycle-identity}, we first
derive an indicator identity and then apply finite cutting.
In oriented Klein coordinates, set
\[
 B=\begin{pmatrix}1&\cdots&1\\x_0&\cdots&x_{n+1}\end{pmatrix},
 \qquad S_i=[x_0,\ldots,\widehat{x_i},\ldots,x_{n+1}].
\]
If $\operatorname{rank}B<n+1$, all $S_i$ are degenerate and
\eqref{eq:hb-cocycle-identity} is immediate.
Assume $\operatorname{rank}B=n+1$ and put
$\alpha_i=(-1)^i\det B_{\widehat i}$, where $B_{\widehat i}$ is obtained
by deleting the column indexed by $i$.  Cofactor expansion and
\eqref{eq:hb-simplex-cocycle} yield
\[
 \ker B=\R\alpha,\qquad \sum_{i=0}^{n+1}\alpha_i=0,\qquad
 (-1)^if(x_0,\ldots,\widehat{x_i},\ldots,x_{n+1})
 =\sgn(\alpha_i)\lambda(S_i).
\]

Choose a finite union $\Sigma$ of affine hyperplanes containing the affine
spans of all subsets of at most $n$ vertices.  For $z\notin\Sigma$, consider
its set of barycentric representations
\[
 I_z=\{t\in\R_{\geq0}^{n+2}:Bt=(1,z)^{\mathsf T}\}.
\]
Since $\ker B=\R\alpha$ and $\sum_i t_i=1$ on $I_z$, a nonempty $I_z$
is a compact segment parallel to $\alpha$.
Every $t\in I_z$ has at most one zero coordinate: two zero coordinates
express $z$ using at most $n$ vertices, contrary to $z\notin\Sigma$.
If $t_i=0$, then $z\in S_i$.  Since the affine hull of every degenerate
$S_i$ is contained in $\Sigma$, this implies $\alpha_i\neq0$.
For sufficiently small nonzero $s$, $t+s\alpha\in I_z$ holds in both
directions if all coordinates of $t$ are positive; if $t_i=0$, it holds
precisely when $s\alpha_i>0$.
Thus a nonempty $I_z$ is nondegenerate, with exactly one zero coordinate
at each endpoint and none in its interior.

When $I_z\neq\varnothing$, orient it in the direction $\alpha$.
The zero coordinate at the initial endpoint has $\alpha_i>0$, and
that at the terminal endpoint has $\alpha_i<0$.
Since $z\in S_i$ precisely when some $t\in I_z$ has $t_i=0$, the
nonzero terms in the following sum are $+1$ and $-1$ if $I_z\neq\varnothing$,
and there are none if $I_z=\varnothing$:
\begin{equation}\label{eq:hb-face-indicator-identity}
 \sum_{i=0}^{n+1}\sgn(\alpha_i)\one_{S_i}(z)=0
 \qquad(z\notin\Sigma).
\end{equation}

Cut $\operatorname{conv}\{x_0,\ldots,x_{n+1}\}$ by the affine hyperplanes
forming $\Sigma$, whose nonempty intersections with $B^n$ are complete totally
geodesic hyperplanes.  Denote the full-dimensional cells by $\mathscr C$.
Each facet hyperplane of a nondegenerate $S_i$ is among these hyperplanes,
since it is the affine hull of $n$ of the vertices.  By \autoref{lem:hb-finite-cuts},
$\lambda(S_i)=\sum_{C\in\mathscr C,\,C\subseteq S_i}\lambda(C)$.
For degenerate $S_i$, both sides vanish by simplicity.
Evaluating \eqref{eq:hb-face-indicator-identity} at an interior point
of each cell verifies \eqref{eq:hb-cocycle-identity}:
\[
\sum_{i=0}^{n+1}(-1)^if(x_0,\ldots,\widehat{x_i},\ldots,x_{n+1})=\sum_{C\in\mathscr C}\Big(\sum_{i:\,C\subseteq S_i}\sgn(\alpha_i)\Big)\lambda(C)=0.\qedhere
\]
\end{proof}

The Klein model lets us apply the following Euclidean statement once
the signed simplex function is smooth.  For an open set $D\subseteq\R^n$,
$C^\infty(D)$ denotes the space of real-valued functions on $D$ whose
partial derivatives of every order exist and are continuous.

\begin{lemma}\label{lem:smooth-simple-density}
For an open convex set $D\subseteq\R^n$, $n\geq1$, let $\lambda$
be a simple valuation on the compact convex Euclidean polytopes contained in $D$.
If its signed simplex function
\[
 f(x_0,\ldots,x_n)
 =\operatorname{sgn}\det(x_1-x_0,\ldots,x_n-x_0)
   \lambda([x_0,\ldots,x_n]),
\]
with value zero on degenerate tuples, is smooth on $D^{n+1}$, then
there is $h\in C^\infty(D)$ such that
\[
 \lambda(P)=\int_P h(x)\dd x
 \qquad(P\subseteq D\text{ a Euclidean polytope}).
\]
\end{lemma}

\begin{proof}
We recover a candidate density from the diagonal derivative, then
show that the error vanishes under subdivision.
Write $F=D_1\cdots D_nf$, where $D_i$ denotes differentiation in the
vertex $x_i$.  The diagonal derivative $\mathcal A_p=F(p,\ldots,p)$
is an alternating $n$-linear form, by alternation in the last $n$ vertices.
Thus, for the standard
basis $e_1,\ldots,e_n$ and the smooth function
$h(p)=n!\,\mathcal A_p(e_1,\ldots,e_n)$,
\[
 \mathcal A_p(v_1,\ldots,v_n)
 =\frac{h(p)}{n!}\det(v_1,\ldots,v_n).
\]
Fix a compact Euclidean convex set $K\subseteq D$ and a Euclidean simplex
$S=[p,p+v_1,\ldots,p+v_n]\subseteq K$ of Euclidean diameter $\delta$.
Since $f$ vanishes whenever one of the last $n$ vertices equals $p$,
repeated application of the fundamental theorem of calculus yields
\[
 f(p,p+v_1,\ldots,p+v_n)
 =\int_{[0,1]^n}F(p,p+t_1v_1,\ldots,p+t_nv_n)
 [v_1,\ldots,v_n]\dd t.
\]
The first derivatives of $F$ are bounded on $K^{n+1}$, so the mean value
estimate in the operator norm gives
\[
 \sup_{t\in[0,1]^n}
 \bigl\|F(p,p+t_1v_1,\ldots,p+t_nv_n)-\mathcal A_p\bigr\|
 \leq C_K\delta.
\]
Subtracting $\mathcal A_p[v_1,\ldots,v_n]$ under the integral yields
\[
 |\lambda(S)-h(p)\operatorname{vol}^{\mathrm E}_n(S)|
 \leq C_K\delta\prod_{i=1}^n\|v_i\|.
\]
Since $|h(x)-h(p)|\leq C_K\delta$ on $S$ and
$n!\operatorname{vol}^{\mathrm E}_n(S)=|\det(v_1,\ldots,v_n)|
\leq\prod_i\|v_i\|$, enlarging $C_K$ yields
\begin{equation}\label{eq:smooth-simple-bound}
 \left|\lambda(S)-\int_Sh(x)\dd x\right|
 \leq C_K\delta\prod_{i=1}^n\|v_i\|
 \leq C_K\delta^{n+1}.
\end{equation}
Estimate~\eqref{eq:smooth-simple-bound} also holds for degenerate simplices by simplicity.

To reduce the representation to simplices, triangulate a full-dimensional
Euclidean polytope $P\subseteq D$ by recursively coning triangulated faces to
relative interior points.  For any triangulation $\mathscr T$ of $P$,
cut $P$ along all facet hyperplanes of its simplices.
The resulting full-dimensional cells $\mathscr C$ lie in unique members
of $\mathscr T$.  Iterating the valuation identity across these Euclidean
hyperplanes, as in the proof of \autoref{lem:hb-finite-cuts}, gives
\[
 \lambda(P)=\sum_{C\in\mathscr C}\lambda(C)
 =\sum_{T\in\mathscr T}
   \sum_{C\in\mathscr C,\,C\subseteq T}\lambda(C)
 =\sum_{T\in\mathscr T}\lambda(T).
\]
Integration has the same additivity, since simplex boundaries have
Lebesgue measure zero.  It therefore suffices to prove the representation
on a nondegenerate simplex $S$.

Write $S=L\Delta+b$, where $L$ is invertible and
$\Delta=\{u\in\R^n:0\leq u_1\leq\cdots\leq u_n\leq1\}$.
For an integer $m\geq1$, divide $[0,1]^n$ into the $m^n$ cubes
$(a+[0,1]^n)/m$, with $a\in\{0,\ldots,m-1\}^n$.
Each cube is subdivided into $n!$ simplices by the inequalities
\[
 0\leq mu_{\sigma(1)}-a_{\sigma(1)}\leq\cdots
 \leq mu_{\sigma(n)}-a_{\sigma(n)}\leq1,
\]
one for each permutation $\sigma$ of $\{1,\ldots,n\}$.
None of these simplices crosses a hyperplane $u_i=u_j$:
when $a_i<a_j$, $u_i\leq u_j$ throughout the cube;
when $a_i=a_j$, their order is prescribed by $\sigma$.
Hence the simplices contained in $\Delta$ triangulate $\Delta$.
Their images under $u\mapsto Lu+b$ form a triangulation $\mathscr T_m$
of $S$ with at most $n!m^n$ members, each of Euclidean diameter at most
$\sqrt n\,\|L\|/m$.
By additivity and \eqref{eq:smooth-simple-bound} with $K=S$,
\[
 \left|\lambda(S)-\int_Sh(x)\dd x\right|
 \leq\sum_{T\in\mathscr T_m}
       \left|\lambda(T)-\int_Th(x)\dd x\right|
 \leq\frac{C_S n!\,(\sqrt n\,\|L\|)^{n+1}}{m}
 \longrightarrow0.
\]
For lower-dimensional $P$ and for $P=\varnothing$, both
$\lambda(P)$ and $\int_P h(x)\dd x$ are zero.
Thus the representation holds for every Euclidean polytope contained in $D$.
\end{proof}

For an isometry-invariant simple valuation, the local bound
\eqref{eq:hb-local-bound} and apex measurability suffice for reflection
averaging of the cocycle identity \eqref{eq:hb-cocycle-identity}.
This yields a smooth simplex function, hence
a density by \autoref{lem:smooth-simple-density}.  Invariance then makes
the density relative to hyperbolic volume constant.

\begin{proposition}\label{prop:hb-simple-volume}
Let $\lambda:\mathcal P(\mathbb H^n)\to\R$ be a simple
isometry-invariant valuation satisfying \eqref{eq:hb-local-bound}
on every compact hyperbolic ball.
If $x\mapsto\lambda([Q,x])$ is measurable with respect to the completion
of hyperbolic volume for every nondegenerate hyperbolic $(n-1)$-simplex $Q$, then
there is $c\in\R$ such that
\[
 \lambda(R)=c\operatorname{vol}_n(R)
 \qquad(R\in\mathcal P(\mathbb H^n)).
\]
\end{proposition}

\begin{proof}
By \eqref{eq:hb-local-bound}, for every compact hyperbolic ball $B$,
\[
 \sup_{x_0,\ldots,x_n\in B}|f(x_0,\ldots,x_n)|
 \leq\sup_{R\in\mathcal P(\mathbb H^n),\,R\subseteq B}|\lambda(R)|
 <\infty.
\]
Fix $x_1,\ldots,x_n$ and use oriented Klein coordinates.
If they are affinely independent, put $Q=[x_1,\ldots,x_n]$.  Then
\[
 f(y,x_1,\ldots,x_n)
 =\sgn\det(x_1-y,\ldots,x_n-y)\,\lambda([Q,y]).
\]
The determinant sign is Borel measurable, and $y\mapsto\lambda([Q,y])$
is measurable by hypothesis.  If the fixed vertices are affinely dependent,
then $\dim[y,x_1,\ldots,x_n]\leq n-1$, so $f(y,x_1,\ldots,x_n)=0$.
Permuting vertices transfers measurability to every variable, since $f$ is alternating.

We first average in a single vertex, using the separate measurability
just established.  Choose $\psi\in C_c^\infty(\R)$ nonnegative and
positive near zero, and set
\[
 Z=\int_{\mathbb H^n}
 \psi\bigl(\cosh\operatorname{dist}_{\mathbb H}(p_*,y)-1\bigr)\dd y\quad\text{ and }\quad
 k(x,y)=Z^{-1}\psi\bigl(\cosh\operatorname{dist}_{\mathbb H}(x,y)-1\bigr),
\]
where $p_*$ is fixed.  In this averaging argument, $\dd y$ denotes
hyperbolic volume, and products of these differentials denote the corresponding product measure.
By compact support and positivity of $\psi$ near zero, $0<Z<\infty$.
The value of $Z$ is independent of $p_*$, since isometries preserve hyperbolic volume.
Therefore
$\int k(x,y)\dd y=1$ and $k(gx,gy)=k(x,y)$ for every isometry $g$.
The identity \cite[Theorem~6.1.1]{Ratcliffe}
\[
 \cosh\operatorname{dist}_{\mathbb H}(x,y)
 =\frac{1-\langle x,y\rangle}
        {\sqrt{(1-\|x\|^2)(1-\|y\|^2)}}
\]
in Klein coordinates shows that $k$ is smooth, including on the diagonal.
Its support satisfies $\operatorname{dist}_{\mathbb H}(x,y)\leq r$
for some $r>0$.

Integrate \eqref{eq:hb-cocycle-identity} for $(y,x_0,\ldots,x_n)$ against
$k(x_0,y)\dd y$.  All terms are absolutely integrable by the support
and local bound \eqref{eq:hb-local-bound}.  For $1\leq i\leq n$, set $F_i(y)=f(y,x_0,\ldots,\widehat{x_i},\ldots,x_n)$.
The affine hull of its $n$ fixed vertices has dimension at most $n-1$;
choose a complete totally geodesic hyperplane containing them and denote
reflection in that hyperplane by $R_i$.
Then $R_i x_0=x_0$ and $F_i(R_i y)=-F_i(y)$ by
\eqref{eq:hb-sign-equivariance}. A change of variables by $R_i$,
using invariance of hyperbolic volume and of the kernel, yields
\[
 \int_{\mathbb H^n} k(x_0,y)F_i(y)\dd y
 =\int_{\mathbb H^n} k(x_0,R_i y)F_i(R_i y)\dd y
 =-\int_{\mathbb H^n} k(x_0,y)F_i(y)\dd y.
\]
Hence this integral is zero. Only the terms omitting $y$ and $x_0$ remain.  By alternation, for $\boldsymbol{x}=(x_0,\ldots,x_n)$ and every $0\leq i\leq n$,
\begin{equation}\label{eq:hb-one-variable-average}
 f(\boldsymbol{x})
 =\int_{\mathbb H^n}k(x_i,y)
       f(x_0,\ldots,x_{i-1},y,x_{i+1},\ldots,x_n)\dd y.
\end{equation}

We next deduce joint continuity from \eqref{eq:hb-one-variable-average}.
For a compact hyperbolic ball $B$, choose a larger compact hyperbolic ball $B'$ containing
its closed $r$-neighborhood in the hyperbolic metric and a bound $M$ for
$|f|$ on $(B')^{n+1}$.
For $\boldsymbol{x},\boldsymbol{x}'\in B^{n+1}$ differing only in the $i$th vertex,
subtract the two instances of \eqref{eq:hb-one-variable-average}:
\[
 f(\boldsymbol{x})-f(\boldsymbol{x}')
 =\int_{B'}\bigl(k(x_i,y)-k(x_i',y)\bigr)
 f(x_0,\ldots,x_{i-1},y,x_{i+1},\ldots,x_n)\dd y.
\]
For arbitrary $\boldsymbol{x},\boldsymbol{x}'\in B^{n+1}$, replace the
vertices one at a time and sum the resulting bounds:
\[
 |f(\boldsymbol{x})-f(\boldsymbol{x}')|
 \leq M\operatorname{vol}_n(B')\sum_{i=0}^n
       \sup_{y\in B'}|k(x_i,y)-k(x_i',y)|
 \qquad(\boldsymbol{x},\boldsymbol{x}'\in B^{n+1}).
\]
By uniform continuity of $k$ on $B\times B'$, the right-hand side tends
to zero as $\boldsymbol{x}'\to\boldsymbol{x}$.
Thus $f$ is jointly continuous and Borel measurable.  For $\boldsymbol{x}\in
B^{n+1}$, write $\boldsymbol{y}=(y_0,\ldots,y_n)$ and
$\dd\boldsymbol{y}=\dd y_0\cdots\dd y_n$.
Since each kernel integrates to one,
$\int|f(\boldsymbol{y})|\prod_i k(x_i,y_i)\dd\boldsymbol{y}\leq M$.
Fubini's theorem permits successive application of \eqref{eq:hb-one-variable-average}:
\begin{equation}\label{eq:hb-joint-average}
 f(x_0,\ldots,x_n)
 =\int_{(\mathbb H^n)^{n+1}}f(y_0,\ldots,y_n)
   \prod_{i=0}^n k(x_i,y_i)\dd y_0\cdots\dd y_n.
\end{equation}
To prove smoothness, consider a tuple of multiindices
$\alpha=(\alpha_0,\ldots,\alpha_n)$ in
Klein coordinates.  The kernel derivatives are uniformly bounded for
$x_i\in B$ and supported in $B'$ as functions of $y_i$.  Hence
\begin{equation}\label{eq:hb-kernel-derivative-bound}
 \left|f(\boldsymbol{y})\prod_{i=0}^n
       \partial_{x_i}^{\alpha_i}k(x_i,y_i)\right|
 \leq M C_\alpha\one_{(B')^{n+1}}(\boldsymbol{y})
  \qquad(\boldsymbol{x}\in B^{n+1}).
\end{equation}
The right-hand side is integrable.  Applying~\eqref{eq:hb-kernel-derivative-bound} to
derivatives of one higher order and using the mean value theorem
provides an integrable bound for each coordinate difference quotient.
Dominated convergence therefore justifies differentiating
\eqref{eq:hb-joint-average} under the integral:
\[
 \partial^\alpha f(\boldsymbol{x})
 =\int_{(\mathbb H^n)^{n+1}}f(\boldsymbol{y})
   \prod_{i=0}^n\partial_{x_i}^{\alpha_i}k(x_i,y_i)\dd\boldsymbol{y}.
\]
Each derivative is continuous by dominated convergence, so
$f\in C^\infty((\mathbb H^n)^{n+1})$.

Applying~\autoref{lem:smooth-simple-density} to Euclidean polytopes
in Klein coordinates gives a smooth $h$ with
$\lambda(P)=\int_P h(x)\dd x$, where $\dd x$ is Euclidean Lebesgue measure.  Set
$q=h/J$, with $J$ as in \eqref{eq:hb-klein-metric-density}, so that
$\lambda(P)=\int_P q\dd\operatorname{vol}_n$.
Hyperbolic isometries preserve both volume and the class $\mathcal P(\mathbb H^n)$.
Thus, for every isometry $g$ and every such $P$,
\[
 \begin{aligned}
 \int_P q(gx)\dd\operatorname{vol}_n(x)
 &=\int_{gP}q(y)\dd\operatorname{vol}_n(y)\\
 &=\lambda(gP)=\lambda(P)=\int_Pq(x)\dd\operatorname{vol}_n(x).
 \end{aligned}
\]
To obtain pointwise invariance of $q$, fix $x\in B^n$.
For sufficiently small $t>0$, the Euclidean cube
$Q_t=x+[-t,t]^n\subset B^n$ is a hyperbolic polytope in the Klein model.
The averages of the continuous function $q\circ g-q$ satisfy
\[
 q(gx)-q(x)
 =\lim_{t\to0^+}\frac1{\operatorname{vol}_n(Q_t)}
   \int_{Q_t}\bigl(q(gy)-q(y)\bigr)\dd\operatorname{vol}_n(y)=0.
\]
Thus $q\circ g=q$ for every isometry $g$.
By transitivity, $q\equiv c$ for some $c\in\R$, and hence
$\lambda(P)=c\operatorname{vol}_n(P)$ for each hyperbolic polytope $P$.
\end{proof}

The reflection cancellation in \eqref{eq:hb-one-variable-average}
uses orientation-reversing isometries.  Invariance under the identity
component alone does not justify this cancellation.

\section{Hyperbolic classifications}
\label{sec:hyperbolic-classifications}
We use \autoref{prop:hb-simple-volume} and dimension reduction to prove the representation and automatic continuity.  Then we determine the coefficient conditions for monotonicity.

\subsection{Representation and automatic continuity}

\begin{proposition}\label{prop:hb-representation}
Let $n\geq1$ and $\mu:\K(\mathbb H^n)\to\R$ be an
isometry-invariant valuation that is continuous or monotone on nonempty
sets.  Then $\mu$ is continuous and has a unique representation
\[
 \mu=a\chi+\sum_{k=0}^{n-1}b_kW_k^{(n)},\qquad
 a,b_0,\ldots,b_{n-1}\in\R.
\]
\end{proposition}

\begin{proof}
We induct on $n$, subtracting a continuous invariant valuation that
agrees with $\mu$ on lower-dimensional sets.
Write $a$ for the common point value of $\mu$.
For $n=1$, take $\nu=a\chi$.
For $n\geq2$, assume the result in dimension $n-1$ and fix a complete
totally geodesic hyperplane $H$.  Every isometry of $H$ extends to
$\mathbb H^n$, so the restriction of $\mu$ satisfies the induction
hypotheses.  Using \eqref{eq:hb-restriction} and $\alpha_{n,k}>0$, choose
$b_1,\ldots,b_{n-1}$ so that
\[
 \nu=a\chi+\sum_{k=1}^{n-1}b_kW_k^{(n)},\qquad
 \left.\nu\right|_{\K(H)}=\left.\mu\right|_{\K(H)}.
\]
In either case, $\nu$ is continuous and invariant by
\autoref{lem:hb-crofton-properties}.
For $n\geq2$, every lower-dimensional $K\in\K(\mathbb H^n)$ can be moved into $H$
by an isometry, so $\mu=\nu$ on such sets, as already holds for $n=1$.
Thus $\lambda=\mu-\nu$ is simple and invariant on $\mathcal P(\mathbb H^n)$.

\autoref{lem:hb-apex} supplies the local bound \eqref{eq:hb-local-bound}
and apex measurability required by \autoref{prop:hb-simple-volume}.
Consequently, for some $b_0\in\R$,
\[
 \mu=\varphi\coloneqq\nu+b_0W_0^{(n)}
 \quad\text{on }\mathcal P(\mathbb H^n),
\]
with $\varphi$ continuous.

To extend the representation to full-dimensional compact hyperbolic convex sets,
fix $K\in\K(\mathbb H^n)$ with $\dim K=n$ and take $R_i^-\subseteq K\subseteq R_i^+$ as in
\autoref{lem:hb-klein-compactness}.
If $\mu$ is continuous, passing to the limit along $R_i^-$ proves
$\mu(K)=\varphi(K)$.  If $\mu$ is monotone on nonempty sets, then
\[
 \varphi(K)=\lim_i\mu(R_i^-)\leq\mu(K)
 \leq\lim_i\mu(R_i^+)=\varphi(K).
\]
On lower-dimensional sets, $W_0^{(n)}=0$ and $\mu=\nu$, so the
representation holds there as well, including at $\varnothing$.
The representation makes $\mu$ continuous, since each summand is
continuous.  The coefficients are unique by the linear independence
in \autoref{lem:hb-crofton-properties}.
\end{proof}

\subsection{Monotonicity and the coefficient cone}

The hyperbolic volume term disappears under restriction to a complete totally geodesic hyperplane.  To determine
its coefficient's sign, we construct nested hyperbolic polytopes whose other
quermassintegral increments are negligible compared with their volume
increment.

\begin{lemma}\label{lem:hb-facet-perturbation}
Let $n\geq2$.  There exist $c,t_0>0$, a full-dimensional hyperbolic polytope
$P\in\mathcal P(\mathbb H^n)$, and hyperbolic polytopes
$P_t\in\mathcal P(\mathbb H^n)$ for $0<t<t_0$ such that $P\subseteq P_t$ and
\[
 W_0^{(n)}(P_t)-W_0^{(n)}(P)\geq ct,\qquad
 \lim_{t\to0^+}
 \frac{W_k^{(n)}(P_t)-W_k^{(n)}(P)}t=0
 \quad(1\leq k<n).
\]
\end{lemma}

\begin{proof}
We bound the hyperbolic volume increment from below by the Euclidean
volume of an added pyramid.  The remaining limits follow by dominated
convergence, using eventual equality of almost every Euclidean orthogonal projection.
In Klein coordinates, fix $0<r<1/\sqrt n$ and set
\[
 P=[-r,r]^{n-1}\times[-r,0],\qquad
 P_t=\operatorname{conv}(P\cup\{te_n\}),\quad0<t<r.
\]
$P$ and $P_t$ lie in the compact hyperbolic ball
$C=\sqrt n\,r\overline B^n\subseteq B^n$.
The closure of $P_t\setminus P$ is the pyramid with base
$[-r,r]^{n-1}\times\{0\}$ and apex $te_n$.
Its height is $t$ and its base area is $(2r)^{n-1}$; since the hyperbolic
volume density satisfies $J\geq1$,
\[
 W_0^{(n)}(P_t)-W_0^{(n)}(P)
 \geq\operatorname{vol}^{\mathrm E}_n(P_t\setminus P)
 =\frac{(2r)^{n-1}}n\,t.
\]

To prove the remaining limits, fix $1\leq k<n$.
For $E\in\Gr(\R^n,k)$, write
$A_E=\pi_{E^\perp}P$ and $A_{E,t}=\pi_{E^\perp}P_t$.
The Crofton formula \eqref{eq:hb-crofton} and its density
\eqref{eq:hb-plane-density} give
\[
 g_t(E)\coloneqq\frac1t\int_{A_{E,t}\setminus A_E}\rho_k(E,u)\dd u,
 \qquad
 \frac{W_k^{(n)}(P_t)-W_k^{(n)}(P)}t
 =\int_{\Gr(\R^n,k)}g_t(E)\dd\nu_k(E).
\]
The projection indicators are jointly Borel measurable by the proof of
\autoref{lem:hb-crofton-properties}, so $g_t$ is measurable by Tonelli's theorem.

For $E\not\subseteq e_n^\perp$, choose $v\in E$ with $v_n<0$.
Then $sv\in\operatorname{int}P$ for sufficiently small $s>0$.
Projecting a Euclidean ball about $sv$ contained in $P$ shows that $A_E$ contains
a neighborhood of $0$ in $E^\perp$.  Consequently,
\[
 A_{E,t}
 =\operatorname{conv}\bigl(A_E\cup\{t\pi_{E^\perp}e_n\}\bigr)
 =A_E\qquad\text{for all sufficiently small }t>0.
\]
Thus $g_t(E)\to0$ for $\nu_k$-almost every $E$ by
\eqref{eq:hb-exceptional-directions}.

To bound $g_t(E)$ uniformly in $E$ and $t$, put $z=-re_n/2$.
Since $z,0\in P$ and $te_n=z+(1+2t/r)(0-z)$, convexity gives
\[
 P\subseteq P_t\subseteq z+(1+2t/r)(P-z).
\]
Write $m=n-k$.  On $A_{E,t}\subseteq\pi_{E^\perp}C$, the density in
\eqref{eq:hb-plane-density} is bounded by
$M_k=\beta_{n,k}(1-nr^2)^{-(n+1)/2}$.
Projecting the preceding inclusion and using homogeneity of Euclidean
volume, we obtain
\[
 0\leq g_t(E)
 \leq M_k\frac{(1+2t/r)^m-1}{t}
       \operatorname{vol}^{\mathrm E}_m(A_E)
 \leq C_k\qquad(0<t<r),
\]
with $C_k$ independent of $E$ and $t$: the mean value theorem bounds
the quotient by $(2m/r)3^{m-1}$ for $0<t<r$, and
$\operatorname{vol}^{\mathrm E}_m(A_E)\leq
\operatorname{vol}^{\mathrm E}_m(\sqrt n\,r\overline B^m)$.
Since $\nu_k$ is a probability measure, dominated convergence proves
the required limit.  The choices $t_0=r$ and $c=(2r)^{n-1}/n$
complete the proof.
\end{proof}

\begin{proposition}\label{prop:hb-coefficient-cone}
For $n\geq1$ and $a,b_0,\ldots,b_{n-1}\in\R$, define
\[
 \varphi=a\chi+\sum_{k=0}^{n-1}b_kW_k^{(n)}.
\]
Then $\varphi$ is monotone on nonempty sets if and only if
$b_0,\ldots,b_{n-1}\geq0$.
It is monotone on $\K(\mathbb H^n)$ if and only if
$a,b_0,\ldots,b_{n-1}\geq0$.
\end{proposition}

\begin{proof}
For monotonicity on nonempty sets, sufficiency follows from \autoref{lem:hb-crofton-properties}: if $b_0,\ldots,b_{n-1}\geq0$ and $\varnothing\neq K\subseteq L$, then
\[
 \varphi(L)-\varphi(K)
 =\sum_{k=0}^{n-1}b_k\bigl(W_k^{(n)}(L)-W_k^{(n)}(K)\bigr)\geq0.
\]

For necessity, suppose that $\varphi$ is monotone on nonempty sets
and induct on $n$.  In dimension one, a nondegenerate segment $I$
and a point $x\in I$ satisfy
\[
 0\leq\varphi(I)-\varphi(\{x\})=b_0\operatorname{vol}_1(I),
\]
so $b_0\geq0$.
For $n\geq2$, restriction to a complete totally geodesic hyperplane
$H\cong\mathbb H^{n-1}$ gives, by \eqref{eq:hb-restriction},
\[
 \left.\varphi\right|_{\K(H)}
 =a\chi+\sum_{k=1}^{n-1}\alpha_{n,k}b_kW_{k-1}^{(n-1)}.
\]
This restriction is monotone on nonempty sets, so induction and
$\alpha_{n,k}>0$ imply $b_1,\ldots,b_{n-1}\geq0$.

To prove $b_0\geq0$, take $P\subseteq P_t$ from
\autoref{lem:hb-facet-perturbation} and write
$\Delta_k(t)=W_k^{(n)}(P_t)-W_k^{(n)}(P)$ for $0\leq k<n$.
Since $\Delta_0(t)\geq ct>0$ and $\Delta_k(t)=o(t)$ for $1\leq k<n$,
\[
 0\leq\frac{\varphi(P_t)-\varphi(P)}{\Delta_0(t)}
 =b_0+\sum_{k=1}^{n-1}b_k\frac{\Delta_k(t)}{\Delta_0(t)}
 \longrightarrow b_0\qquad(t\to0^+).
\]
Thus $b_0\geq0$. It remains to include comparisons with $\varnothing$. Once $\varphi$ is monotone on nonempty sets, every $\varnothing\neq K$ and $x\in K$ satisfy $\varphi(K)\geq\varphi(\{x\})=a$.
Since $\varphi(\varnothing)=0$, the additional comparisons hold if and only if $a\geq0$.
\end{proof}

\begin{proof}[Proof of \autoref{thm:hyperbolic}]
\autoref{prop:hb-representation} gives the unique representation
in \textup{(i)} whenever $\mu$ is continuous or monotone on nonempty sets.
Such linear combinations are continuous by
\autoref{lem:hb-crofton-properties}, proving \textup{(i)} and automatic
continuity in the monotone case.
\autoref{prop:hb-coefficient-cone} then proves \textup{(ii)} and \textup{(iii)}.
\end{proof}

\section*{Acknowledgements}

The author used OpenAI's Codex for discussions of mathematical arguments,
including the measurability and integrability of apex functions,
and for assistance with drafting and revising the manuscript.

\end{document}